\documentclass[11pt]{article}
\usepackage{amssymb,amsmath}
\usepackage[mathscr]{eucal}
\usepackage[cm]{fullpage}
\usepackage[english]{babel}
\usepackage[latin1]{inputenc}
\usepackage{xcolor}
\def\dom{\mathop{\mathrm{Dom}}\nolimits}
\def\im{\mathop{\mathrm{Im}}\nolimits}
\def\id{\mathrm{id}}
\def\N{\mathbb N}
\def\Sym{\mathcal{S}}
\def\I{\mathcal{I}} 

\def\C{\mathcal{C}}
\def\D{\mathcal{D}}

\def\DI{\mathcal{DI}}
\def\OPDI{\mathcal{OPDI}}
\def\ODI{\mathcal{ODI}}
\def\OCI{\mathcal{OCI}}
\def\CI{\mathcal{CI}}
\def\MDI{\mathcal{MDI}}
\def\PODI{\mathcal{PODI}}
\def\POI{\mathcal{POI}}
\def\POPI{\mathcal{POPI}}
\def\PORI{\mathcal{PORI}}

\newtheorem{theorem}{Theorem}[section]
\newtheorem{proposition}[theorem]{Proposition}

\newtheorem{lemma}[theorem]{Lemma}
\newenvironment{proof}{\begin{trivlist}\item[\hskip%
\labelsep{\bf Proof.}]}%
{\qed\rm\end{trivlist}}
\newcommand{\qed}{{\unskip\nobreak
\hfil\penalty50\hskip .001pt \hbox{}
          \nobreak\hfil
         \vrule height 1.2ex width 1.1ex depth -.1ex
           \parfillskip=0pt\finalhyphendemerits=0\medbreak}}

\newcommand{\lastpage}{\addresss}

\newcommand{\addresss}{\small \sf

\noindent{\sc Ilinka Dimitrova},
Department of Mathematics,
Faculty of Mathematics and Natural Science,
South-West University "Neofit Rilski",
2700 Blagoevgrad,
Bulgaria;
e-mail: ilinka\_dimitrova@swu.bg.

\medskip

\noindent{\sc V\'\i tor H. Fernandes},
Center for Mathematics and Applications (NOVA Math)
and Department of Mathematics, 
Faculdade de Ci\^encias e Tecnologia,
Universidade Nova de Lisboa,
Monte da Caparica,
2829-516 Caparica,
Portugal;
e-mail: vhf@fct.unl.pt.

\medskip

\noindent{\sc J\"{o}rg Koppitz},
Institute of Mathematics and Informatics,
Bulgarian Academy of Sciences,
1113 Sofia,
Bulgaria;
e-mail: koppitz@math.bas.bg.

\medskip

\noindent{\sc Teresa M. Quinteiro},
Instituto Superior de Engenharia de Lisboa,
1950-062 Lisboa,
Portugal.
Also:
Center for Mathematics and Applications (NOVA Math),
Faculdade de Ci\^encias e Tecnologia,
Universidade Nova de Lisboa,
Monte da Caparica,
2829-516 Caparica,
Portugal;
e-mail: tmelo@adm.isel.pt.
}

\title{Presentations for three remarkable submonoids of the dihedral inverse monoid on a finite set} 

\author{I. Dimitrova, V\'\i tor H. Fernandes\footnote{This work is funded by national funds through the FCT - Funda\c c\~ao para a Ci\^encia e a Tecnologia, I.P., under the scope of the projects UIDB/00297/2020 and UIDP/00297/2020 (NOVA Math - Center for Mathematics and Applications).}~
J. Koppitz and T.M. Quinteiro\footnote{This work is funded by national funds through the FCT - Funda\c c\~ao para a Ci\^encia e a Tecnologia, I.P., under the scope of the projects UIDB/00297/2020 and UIDP/00297/2020 (NOVA Math - Center for Mathematics and Applications).}
}

\begin{document}

\maketitle

\begin{abstract}
In this paper we consider the inverse submonoids $\OPDI_n$, $\MDI_n$ and $\ODI_n$ 
of the dihedral inverse monoid $\DI_n$ of all orientation-preserving, monotone and order-preserving transformations, respectively.
Our goal is to exhibit presentations for each of these three monoids.
\end{abstract}

\medskip



\section{Introduction and preliminaries}\label{presection} 

Let $\Omega_n$ be a finite set with $n$ elements ($n\in\N$),
say $\Omega_n=\{1,2,\ldots,n\}$. 
Denote by $\Sym_n$ the \textit{symmetric group} on $\Omega_n$,
i.e. the group (under composition of mappings) of all
permutations on $\Omega_n$, 
and by $\I_n$ the \textit{symmetric inverse monoid} on $\Omega_n$, i.e.
the inverse monoid (under composition of partial mappings) of all 
partial permutations on $\Omega_n$.  

From now on, suppose that $\Omega_n$ is a chain, e.g. $\Omega_n=\{1<2<\cdots<n\}$.

An element $\alpha\in\I_n$ is called \textit{order-preserving}
[\textit{order-reversing}] if $x\leqslant y$ implies $x\alpha\leqslant y\alpha$
[$x\alpha\geqslant y\alpha$], for all $x,y \in \dom(\alpha)$.
A partial permutation is said to be \textit{monotone} if it is order-preserving or order-reversing.  
We denote by $\POI_n$
the inverse submonoid of $\I_n$ of all order-preserving
partial permutations and by $\PODI_n$ the inverse submonoid of $\I_n$
of all monotone partial permutations. 

Let $s=(a_1,a_2,\ldots,a_t)$
be a sequence of $t$ ($t\geqslant0$) elements
from the chain $\Omega_n$.
We say that $s$ is \textit{cyclic}
[\textit{anti-cyclic}] if there
exists no more than one index $i\in\{1,\ldots,t\}$ such that
$a_i>a_{i+1}$ [$a_i<a_{i+1}$],
where $a_{t+1}$ denotes $a_1$. 
We also say that $s$ is \textit{oriented} if $s$ is cyclic or $s$ is anti-cyclic. 
Given a partial permutation $\alpha\in\I_n$ such that
$\dom(\alpha)=\{a_1<\cdots<a_t\}$, with $t\geqslant0$, we
say that $\alpha$ is \textit{orientation-preserving}
[\textit{orientation-reversing}, \textit{oriented}] if the sequence of its images
$(a_1\alpha,\ldots,a_t\alpha)$ is cyclic [anti-cyclic, oriented].  
We denote by $\POPI_n$ the inverse submonoid of $\I_n$
of all orientation-preserving partial 
permutations and by $\PORI_n$ the inverse 
submonoid of $\I_n$ of all
oriented partial permutations.

Notice that, by definition, we have $\POI_n\subseteq\PODI_n\subseteq\PORI_n$ and $\POI_n\subseteq\POPI_n\subseteq\PORI_n$.

For a long time, the monoids $\POI_n$, $\PODI_n$, $\POPI_n$ and $\PORI_n$ have aroused interest in the second author, several of his co-authors, as well as various other authors. 
We can certainly say that the structure of these four inverse monoids is well known.
Green's relations, cardinalities, ideals, congruences, generators and ranks, maximal subsemigroups, presentations, automorphisms and endomorphisms, etc., were already determined (see, for example, 
\cite{Araujo&al:2011, 
Delgado&Fernandes:2000,
East:2006, 
Fernandes:2000,
Fernandes:2001,
Fernandes&Gomes&Jesus:2004,
Fernandes&Gomes&Jesus:2005,
Fernandes&Santos:2019,
Ganyushkin&Mazorchuk:2003, 
Li&Fernandes:2023}).
On the other hand, regarding pseudovarieties generated by the families of these semigroups, despite some advances (see, for example, 
\cite{Cowan&Reilly:1995,Fernandes:1997,Fernandes:1998,Fernandes:2001b,Fernandes:2008}), we have to say that much still remains unknown.

\smallskip

Next, let us consider the following permutations of $\Omega_n$ of order $n$ and $2$ (for $n\geqslant2$), respectively:
$$
g=\begin{pmatrix}
1&2&\cdots&n-1&n\\
2&3&\cdots&n&1
\end{pmatrix}
\quad\text{and}\quad
h=\begin{pmatrix}
1&2&\cdots&n-1&n\\
n&n-1&\cdots&2&1
\end{pmatrix}.
$$
It is clear that $g,h\in\PORI_n$.
Moreover, for $n\geqslant3$, $g$ together with $h$ generates the well-known \textit{dihedral group} $\D_{2n}$ of order $2n$ 
(considered as a subgroup of $\Sym_n$). In fact, we have 
$$
\D_{2n}=\langle g,h\mid g^n=1,h^2=1, hg=g^{n-1}h\rangle=\{\id,g,g^2,\ldots,g^{n-1}, h,hg,hg^2,\ldots,hg^{n-1}\}, 
$$
where $\id$ denotes the identity mapping on $\Omega_n$. 
Observe that, 
for $n\in\{1,2\}$, the dihedral group $\D_{2n}$ of order $2n$ 
cannot be embedded in $\Sym_n$. 

Let also $\C_n$ be the \textit{cyclic group} of order $n$ generated by $g$, i.e. 
$
\C_n=\langle g\mid g^n=1\rangle=\{\id,g,g^2,\ldots,g^{n-1}\}.
$

\smallskip 

Now, observe that the elements of $\I_n$ are precisely all restrictions of permutations on $\Omega_n$. 
For $n\geqslant3$, if we consider only restrictions of permutations on $\Omega_n$ that belong to the dihedral group $\D_{2n}$, 
we obtain the inverse submonoid $\DI_n$ of $\I_n$ named in \cite{Fernandes&Paulista:2022sub} by \textit{dihedral inverse monoid on $\Omega_n$}. 
Notice that, it is clear that, for any subgroup $G$ of $\Sym_n$, the set $\I_n(G)$ of all restrictions of elements of $G$ forms an inverse submonoid of $\I_n$ 
whose group of units is precisely $G$. 
In the aforementioned paper, Fernandes and Paulista studied the monoid $\DI_n$ by determining its cardinality and rank as well as descriptions of its Green's relations and, furthermore, presentations for $\DI_n$. Observe that, as $g,h\in\PORI_n$, we have $\DI_n\subseteq\PORI_n$. 

In this paper, we consider 
submonoids of $\DI_n$ that arise when we consider all orientation-preserving, monotone or order-preserving elements of $\DI_n$, 
i.e. the following three inverse submonoids of $\DI_n$: 
$\OPDI_n=\DI_n\cap\POPI_n$, $\MDI_n=\DI_n\cap\PODI_n$ and $\ODI_n=\DI_n\cap\POI_n$. 
Observe that $\ODI_n\subseteq\MDI_n$ and $\ODI_n\subseteq\OPDI_n$. 
These three monoids  were also studied in \cite{Dimitrova&al:2022} by the authors 
who characterized their Green's relations and calculated their cardinals and ranks. 
Here, 
we aim to determine presentations for the monoids $\ODI_n$, $\MDI_n$ and $\OPDI_n$.  
Notice that, as $\ODI_3=\POI_3$, $\OPDI_3=\POPI_3$ and $\MDI_3=\PODI_3$, presentations for $n=3$ are already known 
(see \cite{Fernandes:2000,Fernandes:2001,Fernandes&Gomes&Jesus:2004}). 

\smallskip 

Throughout this paper, we take $n\geqslant4$.  

\smallskip

Until the end of this section, we remember some properties of $\ODI_n$, $\MDI_n$ and $\OPDI_n$, presented by the authors in \cite{Dimitrova&al:2022}, that we will need in the following sections.

\smallskip 

Recall that the \textit{rank} of a partial permutation $\alpha\in\I_n$ 
is the size of $\im(\alpha)$. On the other hand, the \textit{rank} of a monoid $M$ is the minimum size of a generating set of $M$.

\smallskip 

For $X\subseteq\Omega_n$, denote by $\id_X$ the partial identity with domain $X$, i.e. $\id_X=\id|_X$. 
Take 
$$
e_i=\id_{\Omega_n\setminus\{i\}}=
\begin{pmatrix} 1&\cdots&i-1&i+1&\cdots&n\\
1&\cdots&i-1&i+1&\cdots&n
\end{pmatrix}\in\DI_n,
$$
for $1\leqslant i\leqslant n$. Clearly, for $1\leqslant i,j\leqslant n$, we have $e_i^2=e_i$ and $e_ie_j=\id_{\Omega_n\setminus\{i,j\}}=e_je_i$.
More generally, for any $X\subseteq\Omega_n$, we get $\Pi_{i\in X}e_i=\id_{\Omega_n\setminus X}$.

Since the elements of $\DI_n$ are precisely the restrictions of $\D_{2n}$, it is easy to conclude that 
$
\{g,h,e_1,e_2,\ldots,e_n\}
$
is a generating set of $\DI_n$. Moreover, since $g^n=1$ and $e_i=g^{n-i}e_ng^i$ for all $i\in\{1,2,\ldots,n\}$,
it follows that each set $\{g,h,e_i\}$, with $1\leqslant i\leqslant n$, also generates $\DI_n$ 
(see \cite{Dimitrova&al:2022,Fernandes&Paulista:2022sub}).

Notice that $g\in\OPDI_n$, $h\in\MDI_n$ and $e_1,e_2,\ldots,e_n$ are elements of $\ODI_n$, $\MDI_n$ and $\OPDI_n$.

Consider the elements
$$
x=\begin{pmatrix}
1&2&\cdots&n-1\\
2&3&\cdots&n
\end{pmatrix}
\quad\text{and}\quad
y=x^{-1}=\begin{pmatrix}
2&3&\cdots&n\\
1&2&\cdots&n-1
\end{pmatrix} 
$$
of $\ODI_n$ with rank $n-1$ and the elements
$$
x_i=\begin{pmatrix}
1&1+i\\
1&n-i+1
\end{pmatrix}
\quad\text{and}\quad
y_i=x_i^{-1}=\begin{pmatrix}
1&n-i+1\\
1&1+i
\end{pmatrix},
$$
for $1\leqslant i\leqslant\lfloor\frac{n-1}{2}\rfloor$, of $\ODI_n$ with rank $2$.

The following result was proved by the authors in \cite{Dimitrova&al:2022}: 

\begin{proposition}[{\cite[Proposition 4.1 and Theorem 4.3]{Dimitrova&al:2022}}]\label{gensetsrank}
For $n\geqslant4$, 
$$
\{x,y,e_2,\ldots,e_{n-1},x_1,x_2,\ldots,x_{\lfloor\frac{n-1}{2}\rfloor},y_1,y_2,\ldots,y_{\lfloor\frac{n-1}{2}\rfloor}\}, 
$$
$$
\{h,x,e_2,\ldots,e_{\lfloor\frac{n+1}{2}\rfloor},x_1,x_2,\ldots,x_{\lfloor\frac{n-1}{2}\rfloor},y_1,y_2,\ldots,y_{\lfloor\frac{n-1}{2}\rfloor}\}
$$
and 
$$
\{g,e_i,x_1,x_2,\ldots,x_{\lfloor\frac{n-1}{2}\rfloor}\}, 
\quad \mbox{with $1\leqslant i\leqslant n$,}
$$
are generating sets of minimum size of the monoids $\ODI_n$, $\MDI_n$ and $\OPDI_n$, respectively. 
In particular, the monoids $\ODI_n$, $\MDI_n$ and $\OPDI_n$ have ranks  
$n+2\lfloor\frac{n-1}{2}\rfloor$, $2+3\lfloor\frac{n-1}{2}\rfloor$ and $2+\lfloor\frac{n-1}{2}\rfloor$, 
respectively. 
\end{proposition}

Observe that $n+2\lfloor\frac{n-1}{2}\rfloor=2n-\frac{3+(-1)^n}{2}$. 

\medskip 

Finally, we recall also that in \cite[Theorem 2.1 and Theorem 2.3]{Dimitrova&al:2022} the authors showed the following equalities:  
\begin{equation}\label{cardodi} 
|\ODI_n|= 3\cdot2^n+\frac{(n+1)n(n-1)}{6}-\frac{1+(-1)^n}{8}n^2-2n-2 
\end{equation}
and 
\begin{equation}\label{cardmdi} 
|\MDI_n|=
3\cdot2^{n+1}+\frac{(n+1)n(n-1)}{3} -\frac{5+(-1)^n}{4}n^2 -4n-5. 
\end{equation}

\medskip 

For general background on Semigroup Theory and standard notations, we refer to Howie's book \cite{Howie:1995}.

\smallskip

We would like to point out that we made considerable use of computational tools, namely GAP \cite{GAP4}.

\section{A note on presentations}\label{presentations} 

In this section, we recall some notions related to the concept of a monoid presentation.

\smallskip

Let $A$ be an alphabet and consider the free monoid $A^*$ generated by $A$.
The elements of $A$ and of $A^*$ are called \textit{letters} and \textit{words}, respectively.
The empty word is denoted by $1$ and we write $A^+$ to express $A^*\setminus\{1\}$. 
For all $u\in A^*$, the power $u^0$ also denotes the empty word. 
A pair $(u,v)$ of $A^*\times A^*$ is called a
\textit{relation} of $A^*$ and it is usually represented by $u=v$. 
A \textit{monoid presentation} is an ordered pair
$\langle A\mid R\rangle$, where $R\subseteq A^*\times A^*$ is a set of relations of
the free monoid $A^*$. 
A monoid $M$ is said to be
\textit{defined by a presentation} $\langle A\mid R\rangle$ if $M$ is
isomorphic to $A^*/\rho_R$, where $\rho_R$ denotes the smallest
congruence on $A^*$ containing $R$. 
A relation $u=v$ of $A^*$ is said to be a \textit{consequence} of $R$ if $(u,v)\in\rho_R$. 
A set of representatives $W\subseteq A^*$ for the congruence $\rho_R$ is also called a 
\textit{set of forms} for the presentation $\langle A\mid R\rangle$. 
Let $X$ be a generating set of a monoid $M$ and let $\phi: A\longrightarrow M$ be an injective mapping
such that $A\phi=X$.
Let $\varphi: A^*\longrightarrow M$ be the (surjective) homomorphism of monoids that extends $\phi$ to $A^*$.
We say that $X$ satisfies (via $\varphi$) a relation $u=v$ of $A^*$ if $u\varphi=v\varphi$. 
For more details see
\cite{Lallement:1979} or \cite{Ruskuc:1995}.
A direct method to find a presentation for a monoid
is described by the following well-known result (e.g.  see \cite[Proposition 1.2.3]{Ruskuc:1995}).

\begin{proposition}\label{provingpresentation}
Let $M$ be a monoid generated by a set $X$, let $A$ be an alphabet
and let $\phi: A\longrightarrow M$ be an injective mapping
such that $A\phi=X$.
Let $\varphi:A^*\longrightarrow M$ be the (surjective) homomorphism
that extends $\phi$ to $A^*$ and let $R\subseteq A^*\times A^*$.
Then $\langle A\mid R\rangle$ is a presentation for $M$ if and only
if the following two conditions are satisfied:
\begin{enumerate}
\item
The generating set $X$ of $M$ satisfies (via $\varphi$) all relations from $R$;
\item
If $w_1,w_2\in A^*$ are any two words such that
the generating set $X$ of $M$ satisfies (via $\varphi$) the relation $w_1=w_2$ then $w_1=w_2$ is a consequence of $R$.
\end{enumerate}
\end{proposition}

\smallskip

An usual method to find a presentation for a finite monoid
is described by the following result (adapted to the monoid case from
\cite[Proposition 3.2.2]{Ruskuc:1995}). 

\begin{proposition}[Guess and Prove method] \label{ruskuc} 
Let $M$ be a finite monoid generated by a set $X$, 
let $A$ be an alphabet and let $\phi: A\longrightarrow M$ be an injective mapping
such that $A\phi=X$. 
Let $\varphi:A^*\longrightarrow M$ be the (surjective) homomorphism
that extends $\phi$ to $A^*$, let $R\subseteq A^*\times A^*$ 
and $W\subseteq A^*$. Assume that the following conditions are
satisfied:
\begin{enumerate}
\item The generating set $X$ of $M$ satisfies (via $\varphi$) all relations from $R$;
\item For each word $w\in X^*$, there exists a word $w'\in W$ such
that the relation $w=w'$ is a consequence of $R$;
\item $|W|\leqslant |M|$.
\end{enumerate}
Then, $M$ is defined by the presentation $\langle A\mid R\rangle$. 
\end{proposition}

Notice that, if $W$ satisfies the above conditions then, in fact,
$|W|=|M|$ and $W$ is a set of forms for $\langle A\mid R\rangle$.  

\smallskip

Given a presentation for a monoid, another method to find a new
presentation consists in applying Tietze transformations. For a
monoid presentation $\langle A\mid R\rangle$, the
four \emph{elementary Tietze transformations} are:

\begin{description}
\item(T1)
Adding a new relation $u=v$ to $\langle A\mid R\rangle$,
provided that $u=v$ is a consequence of $R$;
\item(T2)
Deleting a relation $u=v$ from $\langle A\mid R\rangle$,
provided that $u=v$ is a consequence of $R\backslash\{u=v\}$;
\item(T3)
Adding a new generating symbol $b$ and a new relation $b=w$, where
$w\in A^*$;
\item(T4)
If $\langle A\mid R\rangle$ possesses a relation of the form
$b=w$, where $b\in A$, and $w\in(A\backslash\{b\})^*$, then
deleting $b$ from the list of generating symbols, deleting the
relation $b=w$, and replacing all remaining appearances of $b$ by
$w$.
\end{description}

The following result is well-known (e.g. see \cite{Ruskuc:1995}):

\begin{proposition} \label{tietze}
Two finite presentations define the same monoid if and only if one
can be obtained from the other by a finite number of elementary
Tietze transformations $(T1)$, $(T2)$, $(T3)$ and $(T4)$.
\end{proposition}

\smallskip 

Next, we recall the process, given in \cite{Fernandes&Gomes&Jesus:2004}, 
to obtain a presentation for a finite monoid $M$ given a presentation for a certain
submonoid of $M$. This method will be applied in Section \ref{presmdi}. 

Let $M$ be a (finite) monoid, let $S$ be a submonoid of $M$ and $y$ be an
element of $M$ such that $y^2=1$. Let us suppose that $M$ is
generated by $S$ and $y$. Let $X=\{x_1,\ldots,x_k\}$ ($k\in\N$) be
a generating set of $S$. Then $Y=X\cup\{y\}$ generates $M$. 
Let $A=\{a_1,\ldots,a_k\}$ be an alphabet, let $b$ be a symbol not belonging to $A$ 
and $B=A\cup\{b\}$. 
Let $\varphi:B^*\longrightarrow M$ be the homomorphism
that extends the mapping $a_i\longmapsto x_i$, for $1\leqslant i\leqslant k$, and $b\longmapsto y$    
to $A^*$ and let $R\subseteq A^*\times A^*$ be such that $\langle A\mid R\rangle$ is a 
presentation for $S$. 
Consider a set of forms $W$ for 
$\langle A\mid R\rangle$ and suppose there exist two subsets $W_1$ and
$W_2$ of $W$ and a word $u_0\in A^*$
such that $W=W_1\cup W_2$ and $u_0$ is a factor of each word in $W_1$. 
Suppose there exist words
$v_0,v_1,\ldots,v_k\in A^*$ such that the following relations over
the alphabet $B$ are satisfied (via $\varphi$) by the generating set $Y$ of $M$:
\begin{description}
\item $(\bar R_1)$ $ba_i=v_ib$, for $1\leqslant i\leqslant k$;
\item $(\bar R_2)$ $u_0b=v_0$.
\end{description}

Observe that the relation (over the alphabet $B$)
\begin{description}
\item $(\bar R_0)$  $b^2=1$
\end{description}
is also satisfied (via $\varphi$) by the generating set $Y$ of $M$, by
hypothesis.

Let $\bar R=R\cup \bar R_0\cup \bar R_1\cup \bar R_2$ and
$\bar{W}=W\cup\{wb\mid w\in W_2\}\subseteq B^*$. 
Then, in \cite{Fernandes&Gomes&Jesus:2004}, Fernandes et al. proved: 

\begin{proposition}[{\cite[Theorem 2.4]{Fernandes&Gomes&Jesus:2004}}] \label{overpresentation}
If $W$ contains the empty word then $\bar{W}$
is a set of forms for the presentation 
$\langle B\mid \bar{R}\rangle$. Moreover, if $|\bar{W}|\leqslant |M|$ then
the monoid $M$ is defined by the presentation 
$\langle B\mid \bar{R}\rangle$. 
\end{proposition}


\section{Presentations for $\ODI_n$}\label{presodi}

In this section, we first determine a presentation for $\ODI_n$ on $2n+\frac{1-(-1)^n}{2}$ generators and, secondly, 
by using Tietze transformations, we deduce another presentation for $\ODI_n$ on $2n-\frac{3+(-1)^n}{2}$ generators. 

\smallskip 

Consider the alphabet $A=\{x,y,e_1,\ldots,e_n,x_1,\ldots,x_{\lfloor\frac{n-1}{2}\rfloor},y_1,\ldots,y_{\lfloor\frac{n-1}{2}\rfloor}\}$ and the set $R$ formed by the following monoid relations:  

\begin{description}

\item $(R_1)$ $e_{i}^{2}=e_{i}$,  for $1\leqslant i\leqslant n$; 

\item $(R_2)$ $xy=e_{n}$ and $yx=e_{1}$;

\item $(R_3)$ $xe_{1}=x$ and $e_{1}y=y$;

\item $(R_4)$ $e_ie_j=e_je_i$, for $1\leqslant i<j\leqslant n$;

\item $(R_5)$ $e_ix=xe_{i+1}$, for $1\leqslant i\leqslant n-1$;

\item $(R_6)$ $x_{i}y_{i}= e_{2}\cdots e_{i}e_{i+2}\cdots e_{n}$, $y_{i}x_{i}=e_{2}\cdots e_{n-i}e_{n-i+2}\cdots e_{n}$, 
for $1\leqslant i \leqslant \lfloor\frac{n-1}{2}\rfloor$; 

\item $(R_7)$ $x_{i}e_{j}=x_{i}$, $e_{j}y_{i}=y_{i}$, 
for $1\leqslant i \leqslant \lfloor\frac{n-1}{2}\rfloor$, $2\leqslant j\leqslant n$ and $j\ne n-i+1$;

\item $(R_8)$ $e_{j}x_{i}=x_{i}$, $y_{i}e_{j}=y_{i}$,
for $1\leqslant i \leqslant \lfloor\frac{n-1}{2}\rfloor$, $2\leqslant j\leqslant n$ and $j\ne i+1$; 

\item $(R_9)$ $e_{1}x_{i}=x_{i}e_{1}=x^{n-2i}e_{n-2i+1}\cdots e_{n-i}e_{n-i+2}\cdots e_n$, $e_{1}y_{i}=y_{i}e_{1}=y^{n-2i}e_1\cdots e_ie_{i+2}\cdots e_{2i}$, 
for $1\leqslant i \leqslant \lfloor\frac{n-1}{2}\rfloor$;

\item $(R_{10})$ $x_{i}e_{n-i+1}=e_{i+1}x_{i} = y_{i}e_{i+1}=e_{n-i+1}y_{i} = e_2\cdots e_n$,
for $1\leqslant i \leqslant \lfloor\frac{n-1}{2}\rfloor$;

\item $(R_{11})$ $xe_2\cdots e_n=e_1\cdots e_n$. 

\end{description}

Observe that $|R|=\frac{1}{2}(5n^2-(1+2(-1)^n)n+(-1)^{n+1}+5)$. 

Throughout Section \ref{presodi}, we represent the congruence $\rho_R$ of $A^*$ by $\approx$. 

\smallskip 

We aim to show that the monoid $\ODI_n$ is defined by the presentation $\langle A \mid R\rangle$. 
To this end, our strategy is to use Proposition \ref{ruskuc} together with the known presentation of a submonoid of $\ODI_n$ 
determined by Fernandes in \cite{Fernandes:2022sub}. 
Therefore, we begin to recall this presentation as well as some other auxiliary results that will also be useful to us here.

\smallskip 

Let us denote by $\CI_n$ the cyclic inverse monoid on $\Omega_n$, i.e. the inverse submonoid of the symmetric 
inverse monoid on $\Omega_n$ consisting of all restrictions of the cyclic group $\C_n$, and by 
$\OCI_n$ the submonoid of $\CI_n$ formed by all order-preserving elements of $\CI_n$.  
Clearly, $\OCI_n$ is also a submonoid of $\ODI_n$. 
Moreover, $\{x,y,e_1,\ldots,e_n\}$ is a generating set of $\OCI_n$ and $|\OCI_n|=3\cdot2^n-2n-2$ \cite[Theorem 1.4]{Fernandes:2022sub}. 
Furthermore, if $U$ is the set of relations $R_1$ to $R_5$ together with relation $R_{11}$ over the alphabet $C=\{x,y,e_1,\ldots,e_n\}$, we have:

\begin{proposition}[{\cite[Theorem 2.16]{Fernandes:2022sub}}]\label{firstpresoci} 
The monoid $\OCI_n$ is defined by the presentation $\langle C \mid U\rangle$ on $n+2$ generators and $\frac{1}{2}(n^2+3n+8)$ relations. 
\end{proposition} 

The following lemma was crucial to prove the above result. Also here, it will be very useful. 

\begin{lemma}[{\cite[Lemma 2.15]{Fernandes:2022sub}}]\label{leftxy}
Let $u\in C^*$. 
Then, there exist $z\in\{x,y\}$, $v\in\{e_1,\ldots, e_n\}^*$ and $0\leqslant r\leqslant n-1$ such that $u=z^rv$ is a consequence of $U$. 
\end{lemma}

Next, recall that the relations 
\begin{equation}\label{powers}
\mbox{$x^je_i=x^j=e_{n-i+1}x^j$ and $e_iy^j=y^j=y^je_{n-i+1}$, for $1\leqslant i\leqslant j\leqslant n$}, 
\end{equation}
are consequences of $U$ (see  \cite[Lemma 2.11]{Fernandes:2022sub}). 

In order to find a set of words $W$ satisfying condition 2 of Proposition \ref{ruskuc} for $\langle A \mid R\rangle$, 
we now present a series of lemmas. 

\smallskip 

The following lemma is easy to deduce from (\ref{powers}) and $R_{10}$: 

\begin{lemma}\label{defw1w2}
Let $1\leqslant i \leqslant \lfloor\frac{n-1}{2}\rfloor$. The following relations are consequences of $R$:
\begin{enumerate}
\item $y^rx_i=y^re_2\cdots e_n$, for $r\geqslant n-i$;
\item $y^ry_i=y^re_2\cdots e_n$, for $r\geqslant i$;
\item $x_ix^s=e_2\cdots e_nx^s$, for $s\geqslant i$;
\item $y_ix^s=e_2\cdots e_nx^s$, for $s\geqslant n-i$. 
\end{enumerate}
\end{lemma}
\begin{proof}
We prove the first relation. For the remaining three ones we can argue in a similar way.

By (\ref{powers}), we have $y^{n-i}\approx y^{n-i}e_{i+1}$ and so, by $R_{10}$, we get $y^{n-i}x_i\approx y^{n-i}e_{i+1}x_i\approx y^{n-i}e_2\cdots e_n$, 
whence $y^rx_i=y^{r-n+i}y^{n-i}x_i\approx y^{r-n+i} y^{n-i}e_2\cdots e_n = y^re_2\cdots e_n$, as required. 
\end{proof}

\begin{lemma}\label{r11r12}
The following relations are consequences of $R$:
\begin{enumerate}
\item $x_ix_j=y_iy_j=e_2\cdots e_n$, for $1\leqslant i ,j\leqslant \lfloor\frac{n-1}{2}\rfloor$;
\item $x_iy_j=y_jx_i=e_2\cdots e_n$, for $1\leqslant i ,j\leqslant \lfloor\frac{n-1}{2}\rfloor$ and $i\neq j$.
\end{enumerate}
\end{lemma}
\begin{proof}
Let $1\leqslant i ,j\leqslant \lfloor\frac{n-1}{2}\rfloor$. 

\smallskip 

Then $j+1\neq n-i+1$. In fact, if $j+1=n-i+1$ then 
$$\textstyle 
1\leqslant j\leqslant \lfloor\frac{n-1}{2}\rfloor\Longrightarrow 1\leqslant n-i\leqslant \lfloor\frac{n-1}{2}\rfloor 
\Longrightarrow \lfloor\frac{n-1}{2}\rfloor < n- \lfloor\frac{n-1}{2}\rfloor \leqslant i, 
$$
which is a contradiction. Notice that $j+1\neq n-i+1$ is equivalent to $i+1\neq n-j+1$. 
Hence, by $R_7$, we have $x_i\approx x_ie_{j+1}$ and $e_{i+1}y_j\approx y_j$. Thus, by $R_1$, $R_4$ and $R_{10}$, we get 
$$
x_ix_j \approx x_ie_{j+1}x_j\approx x_ie_2\cdots e_n \approx x_i e_{n-i+1}e_2\cdots e_{n-i}e_{n-i+2}\cdots e_n \approx 
e_2\cdots e_n e_2\cdots e_{n-i}e_{n-i+2}\cdots e_n \approx e_2\cdots e_n
$$
and 
$$
y_iy_j \approx y_ie_{i+1}y_j\approx e_2\cdots e_n y_j\approx e_2\cdots e_{n-j}e_{n-j+2}\cdots e_n e_{n-j+1} y_j\approx 
e_2\cdots e_{n-j}e_{n-j+2}\cdots e_n e_2\cdots e_n \approx e_2\cdots e_n
$$
(observe that $i,j<n$ implies $n-i+1,n-j+1>1$), which proves property 1. 

\smallskip 

Now, in order to prove property 2, suppose also that $i\neq j$. Then $n-j+1\neq n-i+1$ and $i+1\neq j+1$, 
whence $x_i\approx x_ie_{n-j+1}$, by $R_7$, and $y_je_{i+1}\approx y_j$, by $R_8$. 
Thus, by $R_1$, $R_4$, $R_{10}$ and the above calculations, we obtain 
$$
x_iy_j \approx x_ie_{n-j+1}y_j\approx x_ie_2\cdots e_n\approx e_2\cdots e_n
$$
and 
$$
y_jx_i \approx y_je_{i+1}x_i \approx y_je_2\cdots e_n \approx y_j e_{j+1}e_2\cdots e_je_{j+2}\cdots e_n
\approx e_2\cdots e_n e_2\cdots e_je_{j+2}\cdots e_n \approx e_2\cdots e_n, 
$$
as required.
\end{proof}

\begin{lemma}\label{stepone} 
Let $u\in C^*$ and $1\leqslant i \leqslant \lfloor\frac{n-1}{2}\rfloor$. Then: 
\begin{enumerate}
\item There exists $u'\in C^*$ such that $ux_i\approx u'$ or there exists $0\leqslant r\leqslant n-i-1$ such that $ux_i\approx y^rx_i$;
\item There exists $u'\in C^*$ such that $uy_i\approx u'$ or there exists $0\leqslant r\leqslant i-1$ such that $uy_i\approx y^ry_i$. 
\end{enumerate}
\end{lemma}
\begin{proof}
We prove property 1. The argument for proving property 2 is analogous. 

\smallskip 

Let $z\in\{x,y\}$, $v\in\{e_1,\ldots, e_n\}^*$ and $0\leqslant r\leqslant n-1$ be such that $u\approx z^rv$, by Lemma \ref{leftxy}.  
Since $v\in\{e_1,\ldots, e_n\}^*$, by $R_1$, $R_4$ and $R_8$, we have $vx_i\approx v'x_i$, for some $v'\in\{1,e_1,e_{i+1},e_1e_{i+1}\}$. 

If $v'=e_1^te_{i+1}$, for some $t\in\{0,1\}$, then 
$ux_i\approx z^rv'x_i=z^re_1^te_{i+1}x_i\approx z^re_1^t e_2\cdots e_n\in C^*$, 
by using $R_{10}$. 

If $v'=e_1$ then $ux_i\approx z^rv'x_i=z^re_1x_i\approx z^r x^{n-2i}e_{n-2i+1}\cdots e_{n-i}e_{n-i+2}\cdots e_n\in C^*$, by using $R_9$. 

Now, suppose that $v'=1$. Then $ux_i\approx z^rx_i$. If $r=0$ then there is nothing more to prove. So, suppose also that $r>0$. 

If $z=x$ then $ux_i\approx x^rx_i\approx x^re_1x_i\approx x^r x^{n-2i}e_{n-2i+1}\cdots e_{n-i}e_{n-i+2}\cdots e_n\in C^*$, by $R_3$ and $R_9$. 

If $z=y$ then $ux_i\approx y^rx_i$. If $r\leqslant n-i-1$ then there is nothing more to prove. 
On the other hand, if $r\geqslant n-i$  then, by Lemma \ref{defw1w2}, 
we get $ux_i\approx y^rx_i\approx y^re_2\cdots e_n\in C^*$. 

Thus, the proof of property 1 is complete, as required. 
\end{proof}

\smallskip 

Let $\varphi:A^*\longrightarrow \ODI_n$ be the homomorphism of monoids that extends the mapping 
$A\longrightarrow \ODI_n$ defined by
$$
x\longmapsto x ,\quad y\longmapsto y, \quad e_i\longmapsto e_i, \mbox{~for $1\leqslant i\leqslant n$}, 
\quad x_j\longmapsto x_j~\text{and}~y_j\longmapsto y_j, \mbox{~for $1\leqslant j\leqslant \lfloor\frac{n-1}{2}\rfloor$}. 
$$ 
Notice that we are using the same symbols for the letters of the alphabet $A$ and for the generators of $\ODI_n$,
which simplifies notation and, within the context, will not cause ambiguity. 

\smallskip 

The following subsets of $A^*$, in a sense, were motivated by Lemma \ref{defw1w2}: 
$$
W_1 = \{y^rx_ix^s \mid\mbox{$0\leqslant r,s\leqslant n-1$, $1\leqslant i \leqslant \lfloor\frac{n-1}{2}\rfloor$ and $s+1\leqslant i\leqslant n-r-1$}\} 
$$
and 
$$
W_2 = \{y^ry_ix^s \mid\mbox{$0\leqslant r,s\leqslant n-1$, $1\leqslant i \leqslant \lfloor\frac{n-1}{2}\rfloor$ and $r+1\leqslant i\leqslant n-s-1$}\}.  
$$
Observe that $|W_1\cup W_2|=\frac{1}{6}(n+1)n(n-1) - \frac{1}{8}(1+(-1)^n)n^2$, i.e. the number of order-preserving restrictions of $\{hg^k\mid 0\leqslant k\leqslant n-1\}$ with rank (greater than or) equal to $2$, except those that are also restrictions of $\{g^k\mid 0\leqslant k\leqslant n-1\}$. In fact, $(W_1\cup W_2)\varphi$ is precisely this set of transformations of $\ODI_n$ with rank $2$.
Furthermore, we have: 

\begin{lemma}\label{stepplus}
Let $w\in A^*$. Then, there exists $w'\in C^*\cup W_1\cup W_2$ such that $w\approx w'$.
\end{lemma}
\begin{proof} 
We proceed by induction on $|w|$. 

\smallskip 

If $|w|=1$ then $w\in C\cup W_1\cup W_2$ and so there is nothing to prove. 

\smallskip

As induction hypothesis, assume that the lemma is valid for all words $w\in A^*$ such that $|w|=k\geqslant1$. 

Let $w\in A^*$ be such that $|w|=k+1$. 

If $w\in C^*$ then there is nothing to prove. 
Hence, suppose that $w\in A^*\setminus C^*$. 
Let $u\in A^*$ and $a\in A$ be such that $w=ua$. 
We will consider several cases. 

\smallskip 

\noindent{\sc case 1.} $u\in C^*$. 

Therefore $a\in A\setminus C$. So, by Lemma \ref{stepone}, there exists $u'\in C^*$ such that $ua\approx u'$ 
or there exists $0\leqslant r\leqslant n-i-1$ such that $ua\approx y^rx_i$, with $a=x_i$ for some $1\leqslant i\leqslant \lfloor\frac{n-1}{2}\rfloor$, 
or there exists $0\leqslant r\leqslant i-1$ such that $ua\approx y^ry_i$, with $a=y_i$ for some $1\leqslant i\leqslant \lfloor\frac{n-1}{2}\rfloor$. 
Hence, in any of these three situations, we obtain $w'\in C^*\cup W_1\cup W_2$ such that $w=ua\approx w'$. 

\smallskip 

\noindent{\sc case 2.} $u\in A^*\setminus C^*$. 

Since $|u|=k$ then, by the induction hypothesis, there exists $u'\in C^*\cup W_1\cup W_2$ such that $u'\approx u$. 

\smallskip 

\noindent{\sc case 2.1.} $u'\in C^*$. 

If $a\in C$ then $w=ua\approx u'a\in C^*$. On the other hand, if $a\in A\setminus C$ then, as in {\sc case 1}, there exists 
$w'\in C^*\cup W_1\cup W_2$ such that $u'a\approx w'$ and so such that $w=ua\approx u'a\approx w'$. 

\smallskip 

\noindent{\sc case 2.2.} $u'\in W_1$, i.e. $u'=y^rx_ix^s$, 
for some $0\leqslant r,s\leqslant n-1$, $1\leqslant i \leqslant \lfloor\frac{n-1}{2}\rfloor$ and $s+1\leqslant i\leqslant n-r-1$. 

Then, we have $w=ua\approx u'a= y^rx_ix^sa$. 

\smallskip 

\noindent{\sc case 2.2.1.} $a=e_1$. 

If $s>0$ then, by $R_3$, $x^se_1\approx x^s$ and so $w\approx y^rx_ix^se_1\approx y^rx_ix^s\in W_1$. 
On the other hand, if $s=0$ then, by $R_9$, we have 
$
w\approx y^rx_ie_1\approx  y^r x^{n-2i}e_{n-2i+1}\cdots e_{n-i}e_{n-i+2}\cdots e_n\in C^*
$. 

\smallskip 

\noindent{\sc case 2.2.2.} $a=e_j$, with $2\leqslant j\leqslant n$. 

If $j\leqslant s$ then, by (\ref{powers}), $w\approx y^rx_ix^se_j\approx y^rx_ix^s\in W_1$. 

Now, suppose that $s<j$. Then, by $R_5$, we get $x^se_j\approx e_{j-s}x^s$ and so $w\approx y^rx_ix^se_j\approx y^rx_ie_{j-s}x^s$. 
If $j-s\neq 1$ and $j-s\neq n-i+1$ then, by $R_7$, $w\approx y^rx_ie_{j-s}x^s\approx y^rx_ix^s\in W_1$. 
If $j-s=1$ then, by $R_9$, $w\approx y^rx_ie_1x^s\approx y^r x^{n-2i}e_{n-2i+1}\cdots e_{n-i}e_{n-i+2}\cdots e_n  x^s \in C^*$. 
Finally, if $j-s=n-i+1$ then, by $R_{10}$, $w\approx y^rx_ie_{n-i+1}x^s\approx y^re_2\cdots e_n x^s\in C^*$. 
 
\smallskip 

\noindent{\sc case 2.2.3.} $a=x$. 

If $s+1<i$ then $w\approx y^rx_ix^{s+1}\in W_1$. 
On the other hand, if $s+1\geqslant i$ (in fact $s+1=i$) then, by Lemma \ref{defw1w2}, we get $w\approx y^rx_ix^{s+1}\approx y^re_2\cdots e_nx^{s+1}\in C^*$. 

\smallskip 

\noindent{\sc case 2.2.4.} $a=y$. 

If $s=0$ then, by $R_3$ and $R_9$, 
$w\approx y^rx_iy\approx y^rx_ie_1y\approx y^r x^{n-2i}e_{n-2i+1}\cdots e_{n-i}e_{n-i+2}\cdots e_n y \in C^*$. 
On the other hand, if $s>0$ then, by $R_2$, $R_5$ (noticing that $s-1<n$) and $R_7$ (noticing that $n-s+1>n-i+1$), we have 
$w\approx y^rx_ix^sy\approx  y^rx_ix^{s-1}e_n\approx y^rx_ie_{n-s+1}x^{s-1}\approx y^rx_ix^{s-1}\in W_1$. 

\smallskip 

\noindent{\sc case 2.2.5.} $a=x_j$, with $1\leqslant j \leqslant \lfloor\frac{n-1}{2}\rfloor$.  

If $s=0$ then, by Lemma \ref{r11r12}, we have $w\approx y^rx_ix_j\approx y^r e_2\cdots e_n\in C^*$. 

So, suppose that $s>0$. Then, by $R_3$ and $R_9$, we get 
\begin{equation}\label{225}
w\approx y^rx_ix^sx_j\approx y^rx_ix^se_1x_j\approx y^rx_ix^s x^{n-2j}e_{n-2j+1}\cdots e_{n-j}e_{n-j+2}\cdots e_n. 
\end{equation}
If $s+n-2j\geqslant i$ then from (\ref{225}), by Lemma \ref{defw1w2}, 
we get $w\approx y^re_2\cdots e_n x^{s+n-2j}e_{n-2j+1}\cdots e_{n-j}e_{n-j+2}\cdots e_n\in C^*$. 
On the other hand, suppose that $s+n-2j+1\leqslant i$. 
If $s\geqslant j$ then $n-s+1\leqslant s+n-2j+1\leqslant i$, whence $n-i+1\leqslant s\leqslant i-1$ 
and so $n+2\leqslant 2i\leqslant 2\lfloor\frac{n-1}{2}\rfloor\leqslant n-1$, a contradiction. 
Therefore $s<j$ and so $n-2j+1\leqslant s+n-2j+1\leqslant n-j$. 
Hence, by (\ref{225}), $R_4$, $R_5$ and $R_9$, we obtain  
$$
\begin{array}{rcl}
w& \approx & y^rx_ix^{s+n-2j}e_{s+n-2j+1}e_{n-2j+1}\cdots e_{s+n-2j}e_{s+n-2j+2}\cdots e_{n-j}e_{n-j+2}\cdots e_n \\ 
  & \approx & y^rx_ie_1x^{s+n-2j}e_{n-2j+1}\cdots e_{s+n-2j}e_{s+n-2j+2}\cdots e_{n-j}e_{n-j+2}\cdots e_n \\
   & \approx & y^r  x^{n-2i}e_{n-2i+1}\cdots e_{n-i}e_{n-i+2}\cdots e_n x^{s+n-2j}e_{n-2j+1}\cdots e_{s+n-2j}e_{s+n-2j+2}\cdots e_{n-j}e_{n-j+2}\cdots e_n \in C^*.  
\end{array}
$$

\smallskip 

\noindent{\sc case 2.2.6.} $a=y_j$, with $1\leqslant j \leqslant \lfloor\frac{n-1}{2}\rfloor$.  

If $s=0$ then, by $R_6$ and Lemma \ref{r11r12}, respectively, we have 
$$
w\approx y^rx_iy_j \approx \left\{ 
\begin{array}{ll}
y^r e_2\cdots e_ie_{i+2}\cdots e_n \in C^*& \mbox{if $i=j$}\\
y^r e_2\cdots e_n\in C^* & \mbox{otherwise.}
\end{array} \right. 
$$

Suppose that $s>0$. Then, by $R_3$ and $R_9$, we get 
\begin{equation}\label{226}
w\approx y^rx_ix^sy_j \approx y^rx_ix^se_1y_j \approx y^rx_ix^s y^{n-2j}e_1\cdots e_je_{j+2}\cdots e_{2j}. 
\end{equation} 
Next, by $R_2$, $R_5$ and $R_7$ (noticing that $s+1\leqslant i$ implies $n-s+1>n-i+1$), we have 
$$
x_ix^sy^{n-2j} \approx x_ix^{s-1}e_ny^{n-2j-1} \approx x_i e_{n-s+1}x^{s-1}y^{n-2j-1} \approx x_ix^{s-1}y^{n-2j-1} 
$$
and, repeating this process as long as possible, we obtain 
\begin{equation}\label{226b}
x_ix^sy^{n-2j} \approx 
\left\{ 
\begin{array}{ll}
x_i & \mbox{if $s=n-2j$}\\
x_i x^{s-n+2j}& \mbox{if $s>n-2j$}\\
x_i y^{n-2j-s}& \mbox{if $s<n-2j$}. 
\end{array}\right. 
\end{equation} 

From $(\ref{226})$ and $(\ref{226b})$, we have:  if $s=n-2j$, by $R_9$, 
$$
w\approx y^rx_ie_1e_2\cdots e_je_{j+2}\cdots e_{2j}\approx y^r  x^{n-2i}e_{n-2i+1}\cdots e_{n-i}e_{n-i+2}\cdots e_n  e_2\cdots e_je_{j+2}\cdots e_{2j} \in C^*; 
$$
and, if $s<n-2j$, by $R_3$ and $R_9$,  
$$
\begin{array}{rcl}
w & \approx & y^rx_i y^{n-2j-s}e_1\cdots e_je_{j+2}\cdots e_{2j} \\
& \approx & y^rx_ie_1y^{n-2j-s}e_1\cdots e_je_{j+2}\cdots e_{2j} \\ 
& \approx &  y^r x^{n-2i}e_{n-2i+1}\cdots e_{n-i}e_{n-i+2}\cdots e_n  y^{n-2j-s}e_1\cdots e_je_{j+2}\cdots e_{2j} \in C^*. 
\end{array} 
$$

Now, observe that $s+1\leqslant i$ implies  
$n-s-1\geqslant n-i\geqslant n - \lfloor\frac{n-1}{2}\rfloor= \lfloor\frac{n}{2}\rfloor+1>\lfloor\frac{n-1}{2}\rfloor\geqslant j$, 
i.e. $j<n-s-1$ and so $j>s-n+2j+1$.  
Therefore, if $s>n-2j$, from $(\ref{226})$ and $(\ref{226b})$, we have
$$
\begin{array}{rcl}
w & \approx & y^rx_i x^{s-n+2j}e_1\cdots e_je_{j+2}\cdots e_{2j} \\
& \approx & y^rx_i x^{s-n+2j}  e_{s-n+2j+1} e_1\cdots e_{s-n+2j}e_{s-n+2j+2}\cdots e_je_{j+2}\cdots e_{2j} \\
& \approx & y^r x_ie_1 x^{s-n+2j} e_1\cdots e_{s-n+2j}e_{s-n+2j+2}\cdots e_je_{j+2}\cdots e_{2j} \\
& \approx & y^r x^{n-2i}e_{n-2i+1}\cdots e_{n-i}e_{n-i+2}\cdots e_n x^{s-n+2j} e_1\cdots e_{s-n+2j}e_{s-n+2j+2}\cdots e_je_{j+2}\cdots e_{2j} \in C^*, 
\end{array} 
$$
by $R_4$, $R_5$ and $R_9$, thus completing the proof of the lemma for {\sc case 2.2}. 

\smallskip 

\noindent{\sc case 2.3.} $u'\in W_2$, i.e. $u'=y^ry_ix^s$, 
for some $0\leqslant r,s\leqslant n-1$, $1\leqslant i \leqslant \lfloor\frac{n-1}{2}\rfloor$ and $r+1\leqslant i\leqslant n-s-1$. 

Then, we have $w=ua\approx u'a= y^ry_ix^sa$. 

 If $a\in C$ then proceeding analogously to the cases 2.2.1-2.2.4, we can find $w'\in C^*\cup W_2$ such that $w\approx w'$. 

Thus, it remains to study the case $a\in A\setminus C$, which we divide in two cases. 

\smallskip 

First, let us consider $a=x_j$, for some $1\leqslant j \leqslant \lfloor\frac{n-1}{2}\rfloor$.  

If $s=0$ then, by $R_6$ and Lemma \ref{r11r12}, respectively, we have 
$$
w\approx y^ry_ix_j \approx \left\{ 
\begin{array}{ll}
y^r e_2\cdots e_{n-i}e_{n-i+2}\cdots e_n \in C^*& \mbox{if $i=j$}\\
y^r e_2\cdots e_n\in C^* & \mbox{otherwise.}
\end{array} \right. 
$$

Now, suppose that $s>0$. Then, by $R_3$ and $R_9$, we get 
\begin{equation}\label{23}
w\approx y^ry_ix^sx_j \approx y^ry_ix^se_1x_j \approx y^ry_ix^s  x^{n-2j}e_{n-2j+1}\cdots e_{n-j}e_{n-j+2}\cdots e_n. 
\end{equation}

If $s+n-2j\geqslant n-i$ then 
$$
w\approx y^ry_ix^{s+n-2j}e_{n-2j+1}\cdots e_{n-j}e_{n-j+2}\cdots e_n \approx y^re_2\cdots e_nx^{s+n-2j}e_{n-2j+1}\cdots e_{n-j}e_{n-j+2}\cdots e_n\in C^*,
$$ 
by (\ref{23}) and Lemma \ref{defw1w2}.

So, suppose that $s+n-2j\leqslant n-i-1$, i.e. $i\leqslant n-(s+n-2j)-1$. 

If $s+n-2j=n-j$ then, by (\ref{23}), (\ref{powers}) and $R_5$, 
we have $w\approx y^ry_ix^{s+n-2j}e_{n-j+2}\cdots e_n\approx y^ry_ie_2\cdots e_jx^{s+n-2j}$. 
In addition, if $j\leqslant i$, by $R_8$, we obtain $w\approx y^ry_ie_2\cdots e_jx^{s+n-2j}\approx y^ry_ix^{s+n-2j}\in W_2$. 
Otherwise, by $R_4$, $R_8$ and $R_{10}$, it follows that 
$w\approx y^ry_ie_2\cdots e_ie_{i+2}\cdots e_je_{i+1}x^{s+n-2j} \approx y^ry_ie_{i+1}x^{s+n-2j} \approx y^re_2\cdots e_nx^{s+n-2j}\in C^*$. 

On the other hand, suppose that $s+n-2j\neq n-j$.  Since $s+n-2j+1>n-2j+1$, then $s+n-2j+1\in L=\{n-2j+1,\ldots,n-j,n-j+2,\ldots,n\}$ and so, 
by (\ref{23}), $R_4$, $R_5$ and $R_9$, we have 
$$
\begin{array}{rcl}
w& \approx & y^ry_ix^{s+n-2j}e_{n-2j+1}\cdots e_{n-j}e_{n-j+2}\cdots e_n \\
   & \approx & y^ry_ix^{s+n-2j}e_{s+n-2j+1}\Pi_{t\in L\setminus\{s+n-2j+1\}}e_t\\ 
   & \approx & y^ry_ie_1x^{s+n-2j}\Pi_{t\in L\setminus\{s+n-2j+1\}}e_t\\
   & \approx & y^r y^{n-2i}e_1\cdots e_ie_{i+2}\cdots e_{2i} x^{s+n-2j}\Pi_{t\in L\setminus\{s+n-2j+1\}}e_t\in C^*,  
\end{array} 
$$
which completes the study of this case. 

\smallskip 

Finally, let us move on to our last case by considering $a=y_j$, for some $1\leqslant j \leqslant \lfloor\frac{n-1}{2}\rfloor$.  

If $s=0$ then, by Lemma \ref{r11r12}, we have $w\approx y^ry_iy_j\approx y^r e_2\cdots e_n\in C^*$. 

So, suppose that $s>0$. Then, by $R_3$ and $R_9$, we get 
\begin{equation}\label{23b}
w\approx y^ry_ix^sy_j \approx y^ry_ix^se_1y_j \approx y^ry_ix^s y^{n-2j}e_1\cdots e_je_{j+2}\cdots e_{2j}. 
\end{equation} 
Next, by $R_2$, $R_5$ and $R_8$ (noticing that $i\leqslant n-s-1$ implies $i+1 < n-s+1$), we have 
$$
y_ix^sy^{n-2j} \approx y_ix^{s-1}e_ny^{n-2j-1} \approx y_i e_{n-s+1}x^{s-1}y^{n-2j-1} \approx y_ix^{s-1}y^{n-2j-1}. 
$$
By repeating this process as long as possible, we obtain 
\begin{equation}\label{23c}
y_ix^sy^{n-2j} \approx 
\left\{ 
\begin{array}{ll}
y_i & \mbox{if $s=n-2j$}\\
y_i x^{s-n+2j}& \mbox{if $s>n-2j$}\\
y_i y^{n-2j-s}& \mbox{if $s<n-2j$}. 
\end{array}\right. 
\end{equation} 

If $s=n-2j$, by (\ref{23b}), (\ref{23c}) and $R_9$, we have 
$$
w\approx y^ry_ie_1e_2\cdots e_je_{j+2}\cdots e_{2j}\approx y^r  y^{n-2i}e_1\cdots e_ie_{i+2}\cdots e_{2i}  e_2\cdots e_je_{j+2}\cdots e_{2j} \in C^*. 
$$
On the other hand, if $s<n-2j$, by (\ref{23b}), (\ref{23c}), $R_3$ and $R_9$, we get 
$$
\begin{array}{rcl}
w & \approx & y^ry_i y^{n-2j-s}e_1\cdots e_je_{j+2}\cdots e_{2j} \\
& \approx & y^ry_ie_1y^{n-2j-s}e_1\cdots e_je_{j+2}\cdots e_{2j} \\ 
& \approx &  y^r y^{n-2i}e_1\cdots e_ie_{i+2}\cdots e_{2i} y^{n-2j-s}e_1\cdots e_je_{j+2}\cdots e_{2j} \in C^*. 
\end{array} 
$$

Now, suppose that $s>n-2j$. 

In addition, suppose first that $s-n+2j=j$. Then, by (\ref{23b}), (\ref{23c}), (\ref{powers}) and $R_5$, we obtain 
$$
w\approx y^ry_ix^{j}e_1\cdots e_je_{j+2}\cdots e_{2j}\approx y^ry_ix^{j}e_{j+2}\cdots e_{2j} 
\approx y^ry_ie_2\cdots e_jx^{j}. 
$$
Since $i\leqslant n-s-1 < n-s=j$, then $w\approx y^ry_ie_{i+1}x^{j}\approx y^re_2\cdots e_nx^{j}\in C^*$, by $R_8$, $R_4$ and $R_{10}$.  

Secondly, suppose that $s-n+2j\neq j$. 
Since $s-n<0$, then $s-n+2j+1\leqslant 2j$ and so $s-n+2j+1\in K=\{1,\ldots,j,j+2,\ldots 2j\}$. 
Hence, by (\ref{23b}), (\ref{23c}), $R_4$, $R_5$ and $R_9$, we get 
$$
\begin{array}{rcl}
w&\approx& y^ry_ix^{s-n+2j}e_1\cdots e_je_{j+2}\cdots e_{2j} \\
&\approx& y^ry_ix^{s-n+2j}e_{s-n+2j+1}\Pi_{t\in K\setminus\{s-n+2j+1\}}  e_t \\
&\approx& y^ry_ie_1x^{s-n+2j}\Pi_{t\in K\setminus\{s-n+2j+1\}}  e_t \\ 
&\approx& y^r y^{n-2i}e_1\cdots e_ie_{i+2}\cdots e_{2i}  x^{s-n+2j}\Pi_{t\in K\setminus\{s-n+2j+1\}}  e_t \in C^*. 
\end{array} 
$$

Therefore, we have exhausted all possible cases, completing the proof of the lemma.
\end{proof}

Now, let us choose a set of forms $W_0$ for the presentation $\langle C \mid U\rangle$. 
Then, for each $w\in C^*$, there exists (a unique) $w'\in W_0$ such that $w'\rho_U w$. Moreover, as 
the monoid $\OCI_n$ is defined by the presentation $\langle C \mid U\rangle$, by Proposition \ref{firstpresoci}, 
we have $|W_0|=|\OCI_n|=3\cdot2^n-2n-2$. 

Let $W=W_0\cup W_1\cup W_2$. Then, by (\ref{cardodi}), 
$$
|W|=|W_0|+|W_1\cup W_2|=3\cdot2^n-2n-2+\frac{1}{6}(n+1)n(n-1) - \frac{1}{8}(1+(-1)^n)n^2=|\ODI_n|
$$
and, by Lemma \ref{stepplus}, for each word $w\in A^*$, there exists $w'\in W$ such that $w\approx w'$. 

On the other hand, it is a routine matter to check:  

\begin{lemma}\label{genreloci}
The generating set $\{x,y,e_1,e_2,\ldots,e_n,x_1,\ldots,x_{\lfloor\frac{n-1}{2}\rfloor},y_1,\ldots,y_{\lfloor\frac{n-1}{2}\rfloor}\}$ 
of $\ODI_n$ satisfies (via $\varphi$) all relations from $R$. 
\end{lemma}

Therefore, the conditions of Proposition \ref{ruskuc} are satisfied and so we have: 

\begin{theorem}\label{firstpresoc} 
The monoid $\ODI_n$ is defined by the presentation $\langle A \mid R\rangle$ on $2n+\frac{1-(-1)^n}{2}$ generators and 
$\frac{1}{2}(5n^2-(1+2(-1)^n)n+(-1)^{n+1}+5)$ relations. 
\end{theorem} 

\smallskip 

Next, by using Tietze transformations and applying Proposition \ref{tietze}, we deduce from the above presentation for $\ODI_n$ 
a new presentation on a minimum size set of generators of $\ODI_n$ given in Proposition \ref{gensetsrank}. 

Let us consider the alphabet 
$$
B=\{x,y,e_2,\ldots,e_{n-1},x_1,x_2,\ldots,x_{\lfloor\frac{n-1}{2}\rfloor},y_1,y_2,\ldots,y_{\lfloor\frac{n-1}{2}\rfloor}\}=A\setminus\{e_1,e_n\}.  
$$ 
Basically, we first apply T4 with each of the relations $R_2$ and then, of the resulting relations, we eliminate the trivial ones and some deduced from others.  

This procedure was applied in \cite{Fernandes:2022sub} to the set $U$ of relations ($R_1$ to $R_5$ together with relation $R_{11}$) 
on the alphabet $C=\{x,y,e_1,\ldots,e_n\}$, having resulted in the following set of $\frac{1}{2}(n^2+3n)$ monoid relations (which we can consider on the alphabet $B$): 
\begin{description}

\item $(V_1)$ $e_i^2=e_i$, for $2\leqslant i\leqslant n-1$; 

\item $(V_2)$ $xyx=x$ and $yxy=y$; 

\item $(V_3)$ $yx^2y=xy^2x$;  

\item $(V_4)$ $e_ie_j=e_je_i$, for $2\leqslant i<j\leqslant n-1$; 

\item $(V_5)$ $xye_i=e_ixy$ and $yxe_i=e_iyx$, for $2\leqslant i\leqslant n-1$; 

\item $(V_6)$ $xe_{i+1}=e_ix$, for $2\leqslant i\leqslant n-2$; 

\item $(V_7)$ $x^2y=e_{n-1}x$ and $yx^2=xe_2$; 

\item $(V_8)$ $yxe_2\cdots e_{n-1}xy=xe_2\cdots e_{n-1}xy$.

\end{description}

\smallskip 

Performing the same procedure to relations $R_6$ to $R_{10}$ on the alphabet $A$, we may routinely obtain the following 
$2n^2-(2+(-1)^n)n-\frac{1}{2}(3+(-1)^n)$ monoid relations on the alphabet $B$: 
\begin{description}

\item $(V_9)$ $x_{i}y_{i}= e_{2}\cdots e_{i}e_{i+2}\cdots e_{n-1}xy$, for $1\leqslant i \leqslant \lfloor\frac{n-1}{2}\rfloor$; \\ 
$y_{1}x_{1}=e_{2}\cdots e_{n-1}$; 
$y_{i}x_{i}=e_{2}\cdots e_{n-i}e_{n-i+2}\cdots e_{n-1}xy$, for $2\leqslant i \leqslant \lfloor\frac{n-1}{2}\rfloor$; 

\item $(V_{10})$ $x_{i}e_{j}=x_{i}$ and $e_{j}y_{i}=y_{i}$, 
for $1\leqslant i \leqslant \lfloor\frac{n-1}{2}\rfloor$, $2\leqslant j\leqslant n-1$ and $j\ne n-i+1$; 

\item $(V_{11})$ $e_{j}x_{i}=x_{i}$ and $y_{i}e_{j}=y_{i}$,
for $1\leqslant i \leqslant \lfloor\frac{n-1}{2}\rfloor$, $2\leqslant j\leqslant n-1$ and $j\ne i+1$; 

\item $(V_{12})$ $x_{i}xy=xyx_i=x_i$ and $xyy_{i}=y_ixy=y_{i}$, for $2\leqslant i \leqslant \lfloor\frac{n-1}{2}\rfloor$; $xyx_1=x_1$ and $y_1xy=y_1$; 

\item $(V_{13})$ $yxx_{1}=x_{1}yx=x^{n-2}e_{n-1}$; 
$yxx_{i}=x_{i}yx=x^{n-2i}e_{n-2i+1}\cdots e_{n-i}e_{n-i+2}\cdots e_{n-1}xy$, for $2\leqslant i \leqslant \lfloor\frac{n-1}{2}\rfloor$; 
$yxy_{i}=y_{i}yx=y^{n-2i+1}xe_2\cdots e_ie_{i+2}\cdots e_{2i}$, for $1\leqslant i \leqslant \lfloor\frac{n-1}{2}\rfloor$; 

\item $(V_{14})$ $x_{1}xy=e_{2}x_{1} = y_{1}e_{2}=xyy_{1} = e_2\cdots e_{n-1}xy$; 
$x_{i}e_{n-i+1}=e_{i+1}x_{i} = y_{i}e_{i+1}=e_{n-i+1}y_{i} = e_2\cdots e_{n-1}xy$, for $2\leqslant i \leqslant \lfloor\frac{n-1}{2}\rfloor$. 
\end{description}

Thus, defining $V$ as the set of monoid relations on the alphabet $B$ consisting 
of relations $V_1$ to $V_{14}$, 
we have: 

\begin{theorem}\label{rankpresodi} 
The monoid $\ODI_n$ is defined by the presentation $\langle B\mid V\rangle$ on $2n-\frac{3+(-1)^n}{2}$ generators 
and $\frac{1}{2}(5n^2-(1+2(-1)^n)n+(-1)^{n+1}-3)$ relations. 
\end{theorem}

\section{Presentations for $\MDI_n$}\label{presmdi}

We begin this section by determining a presentation for $\MDI_n$ on $2n-\frac{1+(-1)^n}{2}$ generators. 
For this purpose, we will apply Proposition \ref{overpresentation}. 
Next, by using Tietze transformations, 
we deduce another presentation for $\MDI_n$ on $2+3\lfloor\frac{n-1}{2}\rfloor$ generators. 

\smallskip 

Let us consider the alphabet $\bar B=B\cup\{h\}=\{h,x,y,e_2,\ldots,e_{n-1},x_1,\ldots,x_{\lfloor\frac{n-1}{2}\rfloor},y_1,\ldots,y_{\lfloor\frac{n-1}{2}\rfloor}\}$
and let $\bar\varphi:\bar B^*\longrightarrow \MDI_n$ be the homomorphism of monoids that extends the mapping 
$\bar B\longrightarrow \MDI_n$
defined by
$$
h\longmapsto h ,\quad x\longmapsto x ,\quad y\longmapsto y, \quad e_i\longmapsto e_i, \mbox{~for $2\leqslant i\leqslant n-1$}, 
\quad x_j\longmapsto x_j~\text{and}~y_j\longmapsto y_j, \mbox{~for $1\leqslant j\leqslant \lfloor\frac{n-1}{2}\rfloor$}. 
$$
Notice that  $\bar B\bar\varphi$ is a generating set of $\MDI_n$ with $2n-\frac{1+(-1)^n}{2}$ elements. 

Next, observe that $\emptyset=e_1e_2\cdots e_{n-1}e_n$,  $\binom{1}{1}=e_2\cdots e_{n-1}e_n$ and 
$$\textstyle 
\binom{i}{j}=\binom{i}{1}\binom{1}{1}\binom{1}{j} = e_{i+1}\cdots e_n y^{i-1} e_2\cdots e_{n-1}e_n x^{j-1}e_{j+1}\cdots e_n, 
$$
for $1\leqslant i,j\leqslant n$. 
Therefore, let $u_0=e_2\cdots e_{n-1}xy\in B^*$ and 
let $W'_1$ be the subset of $B^*$ formed by the following $1+n^2$ words: 
\begin{enumerate}
\item $yxu_0$;
\item $e_{i+1}\cdots e_{n-1}x y^i u_0 x^{j-1}e_{j+1}\cdots e_{n-1}xy$, for $1\leqslant i,j\leqslant n-1$; 
\item $e_{i+1}\cdots e_{n-1}x y^i u_0 x^{n-1}$, for $1\leqslant i\leqslant n-1$; 
\item $y^{n-1} u_0 x^{j-1}e_{j+1}\cdots e_{n-1}xy$, for $1\leqslant j\leqslant n-1$; and 
\item $y^{n-1}u_0x^{n-1}$. 
\end{enumerate}
Then $\bar\varphi$ is a bijection from $W'_1$ onto $\{\emptyset\}\cup\{\binom{i}{j}\mid 1\leqslant i,j\leqslant n\}$ 
and so we can choose a set of forms $W'$ for the presentation $\langle B\mid V\rangle$ of $\ODI_n$
such that $W'$ contains the empty word and $W'_1\subset W'$. 
Let $W'_2=W'\setminus W'_1$ and consider the subset 
$\bar{W}=W'\cup\{wh\mid w\in W'_2\}$ of $\bar B^*$. 

Notice that $|\bar W|=|W'| + |W'_2|=|W'| + |W'| -|W'_1| =2|\ODI_n|-n^2-1=|\MDI_n|$, by (\ref{cardodi}) and (\ref{cardmdi}). 

\smallskip

Let $1\leqslant i\leqslant \lfloor\frac{n-1}{2}\rfloor$. Then 
$$
hx_ih = \begin{pmatrix} n-i & n\\ i & n \end{pmatrix} = y^{n-i-1}x_ix^{i-1}
\quad\text{and}\quad
hy_ih = \begin{pmatrix} i & n\\ n-i & n \end{pmatrix} = y^{i-1}y_ix^{n-i-1}. 
$$
On the other hand, we also have 
$$
hxh=y, \quad hyh=x, \quad he_i h=e_{n-i+1}, \mbox{~for $1\leqslant i\leqslant n$}, \quad\text{and}\quad 
e_2\cdots e_nh=\binom{1}{n}=x^{n-1}. 
$$
Therefore, the monoid relations (on the alphabet $\bar B$)
$$
h^2=1, \quad hx=yh, \quad hy=xh, \quad 
\mbox{$he_i =e_{n-i+1}h$, for $2\leqslant i\leqslant n-1$},
$$
$$
\mbox{$hx_i=y^{n-i-1}x_ix^{i-1}h$ and $hy_i=y^{i-1}y_ix^{n-i-1}h$, for $1\leqslant i\leqslant \lfloor\frac{n-1}{2}\rfloor$,} 
\quad\text{and}\quad 
e_2\cdots e_{n-1}xyh=x^{n-1}
$$
are all satisfied (via $\bar\varphi$) by the generating set 
$\{h,x,y,e_2,\ldots,e_{n-1},x_1,\ldots,x_{\lfloor\frac{n-1}{2}\rfloor},y_1,\ldots,y_{\lfloor\frac{n-1}{2}\rfloor}\}$
of $\MDI_n$. 

Now, let $\bar V$ be the set of monoid relations $V$ (relations $V_1$ to $V_{14}$ considered on the alphabet $\bar B$) together with 
the following $n+\lfloor\frac{n+1}{2}\rfloor+\frac{1-(-1)^n}{2}$ monoid relations on the alphabet $\bar B$: 
\begin{description}

\item $(\bar V_0)$ $h^2=1$; 

\item $(\bar V_1)$  $hx=yh$; $he_i =e_{n-i+1}h$, for $2\leqslant i\leqslant \lfloor\frac{n+1}{2}\rfloor$; \\ 
$hx_i=y^{n-i-1}x_ix^{i-1}h$ and $hy_i=y^{i-1}y_ix^{n-i-1}h$, for $1\leqslant i\leqslant \lfloor\frac{n-1}{2}\rfloor$; 

\item $(\bar V_2)$ $e_2\cdots e_{n-1}xyh=x^{n-1}$. 
\end{description}

Since the relation $hy=xh$ is a consequence of $h^2=1$ and $hx=yh$ 
and the relations $he_i =e_{n-i+1}h$, with $2\leqslant i\leqslant n-1$, are consequences of 
$h^2=1$ and $he_i =e_{n-i+1}h$, with $2\leqslant i\leqslant \lfloor\frac{n+1}{2}\rfloor$, by Proposition \ref{overpresentation}, 
we conclude that: 

\begin{theorem}\label{firstpresmdi} 
The monoid $\MDI_n$ is defined by the presentation $\langle \bar B\mid \bar V\rangle$ on $2n-\frac{1+(-1)^n}{2}$ generators 
and $\frac{1}{2}(5n^2+(2-2(-1)^n)n-\frac{3+5(-1)^n}{2})$ relations. 
\end{theorem} 

\smallskip 

Next, like in Section \ref{presodi}, by using Tietze transformations and applying Proposition \ref{tietze}, we deduce from the presentation $\langle \bar B\mid \bar V\rangle$ 
of $\MDI_n$ a new presentation on a minimum size set of generators of $\MDI_n$ provided by Proposition \ref{gensetsrank}. 
So, consider the alphabet 
$$
\bar B'=\{h,x,e_2,\ldots,e_{\lfloor\frac{n+1}{2}\rfloor},x_1,x_2,\ldots,x_{\lfloor\frac{n-1}{2}\rfloor},y_1,y_2,\ldots,y_{\lfloor\frac{n-1}{2}\rfloor}\}. 
$$ 
Hence, since $y=hxh$ and $e_i =he_{n-i+1}h$, for $\lfloor\frac{n+1}{2}\rfloor+1\leqslant i\leqslant n-1$ (as transformations), 
we can apply T1 by adding the relations  $y=hxh$ and $e_i =he_{n-i+1}h$, for $\lfloor\frac{n+1}{2}\rfloor+1\leqslant i\leqslant n-1$. 
Next, we apply T4 with each of these relations and then, of the resulting relations, we eliminate the trivial ones and some deduced from others.  
Performing this procedure to $\bar V$, we may routinely obtain the following set $\bar V'$ 
of $2n^2+\frac{7-(-1)^n}{4}n-2(-1)^n-1$ monoid relations on the alphabet $\bar B'$: 
\begin{description}

\item $(V'_1)$ $e_i^2=e_i$, for $2\leqslant i\leqslant \lfloor\frac{n+1}{2}\rfloor$; 

\item $(V'_2)$ $(xh)^2x=x$; 

\item $(V'_3)$ $xhx^2hxh=h xhx^2hx$;

\item $(V'_4)$ $e_ie_j=e_je_i$, for $2\leqslant i<j\leqslant \lfloor\frac{n+1}{2}\rfloor$; 
$e_ihe_jh=he_jhe_i$, for $2\leqslant i \leqslant \lfloor\frac{n+1}{2}\rfloor$ and $2\leqslant j \leqslant \lfloor\frac{n}{2}\rfloor$;

\item $(V'_5)$ $(xh)^2e_i=e_i(xh)^2$ and $(hx)^2e_i=e_i(hx)^2$, for $2\leqslant i\leqslant \lfloor\frac{n+1}{2}\rfloor$; 

\item $(V'_6)$ $xe_{i+1}=e_ix$, for $2\leqslant i\leqslant \lfloor\frac{n-1}{2}\rfloor$; 
$xhe_{\lfloor\frac{n}{2}\rfloor}h=e_{\lfloor\frac{n+1}{2}\rfloor}x$; 
$xhe_ih=he_{i+1}hx$, for $2\leqslant i\leqslant \lfloor\frac{n-2}{2}\rfloor$; 

\item $(V'_7)$ $x(xh)^2=he_2hx$ and $(hx)^2x=xe_2$; 

\item $(V'_8)$ $(hx)^2 e_2\cdots e_{\lfloor\frac{n+1}{2}\rfloor} h e_2\cdots e_{\lfloor\frac{n}{2}\rfloor}(hx)^2 = 
x e_2\cdots e_{\lfloor\frac{n+1}{2}\rfloor} h e_2\cdots e_{\lfloor\frac{n}{2}\rfloor}(hx)^2$; 

\item $(V'_9)$ $x_iy_i=e_2\cdots e_i e_{i+2}\cdots e_{\lfloor\frac{n+1}{2}\rfloor} h e_2\cdots e_{\lfloor\frac{n}{2}\rfloor} h (xh)^2$, 
for  $1\leqslant i\leqslant \lfloor\frac{n-1}{2}\rfloor$; 
$y_1x_1=e_2\cdots e_{\lfloor\frac{n+1}{2}\rfloor} h e_2\cdots e_{\lfloor\frac{n}{2}\rfloor} h$; 
$y_ix_i=e_2\cdots e_{\lfloor\frac{n+1}{2}\rfloor} h e_2\cdots e_{i-1} e_{i+1}\cdots e_{\lfloor\frac{n}{2}\rfloor} h(xh)^2$, 
for  $2\leqslant i\leqslant \lfloor\frac{n-1}{2}\rfloor$; 

\item $(V'_{10})$ $x_{i}e_{j}=x_{i}$ and $e_{j}y_{i}=y_{i}$, 
for $1\leqslant i \leqslant \lfloor\frac{n-1}{2}\rfloor$ and $2\leqslant j\leqslant \lfloor\frac{n+1}{2}\rfloor$; \\
$x_{i}he_{j}h=x_{i}$ and $he_{j}hy_{i}=y_{i}$, 
for $1\leqslant i \leqslant \lfloor\frac{n-1}{2}\rfloor$, $2\leqslant j\leqslant \lfloor\frac{n}{2}\rfloor$ and $j\neq i$; 

\item $(V'_{11})$ $e_{j}x_{i}=x_{i}$ and $y_{i}e_{j}=y_{i}$,
for $1\leqslant i \leqslant \lfloor\frac{n-1}{2}\rfloor$, $2\leqslant j\leqslant \lfloor\frac{n+1}{2}\rfloor$ and $j\ne i+1$; \\ 
$he_{j}hx_{i}=x_{i}$ and $y_{i}he_{j}h=y_{i}$,
for $1\leqslant i \leqslant \lfloor\frac{n-1}{2}\rfloor$ and $2\leqslant j\leqslant \lfloor\frac{n}{2}\rfloor$; 

\item $(V'_{12})$ $x_{i}(xh)^2=(xh)^2x_i=x_i$ and $(xh)^2y_{i}=y_i(xh)^2=y_{i}$, for $2\leqslant i \leqslant \lfloor\frac{n-1}{2}\rfloor$; 
$(xh)^2x_1=x_1$ and $y_1(xh)^2=y_1$;

\item $(V'_{13})$ $(hx)^2x_1 = x_1(hx)^2 = x^{n-2}he_2h$; \\ 
$(hx)^2x_i = x_i(hx)^2 = 
x^{n-2i} e_{n-2i+1}\cdots e_{\lfloor\frac{n+1}{2}\rfloor} h e_2\cdots e_{i-1}e_{i+1}\cdots e_{\lfloor\frac{n}{2}\rfloor} h (xh)^2$, 
for $2\leqslant i \leqslant \lfloor\frac{n-1}{2}\rfloor$; \\ 
$(hx)^2y_i = y_i(hx)^2 = h x^{n-2i+1} h x e_2\cdots e_ie_{i+2}\cdots e_{\lfloor\frac{n+1}{2}\rfloor} h e_{n-2i+1}\cdots e_{\lfloor\frac{n}{2}\rfloor}h$, 
for $1\leqslant i \leqslant \lfloor\frac{n-1}{2}\rfloor$; 

\item $(V'_{14})$ $x_{1}(xh)^2=e_{2}x_{1} = y_{1}e_{2}=(xh)^2y_{1} = e_2\cdots e_{\lfloor\frac{n+1}{2}\rfloor} h e_2\cdots e_{\lfloor\frac{n}{2}\rfloor}h(xh)^2$; \\
$x_{i}he_ih=e_{i+1}x_{i} = y_{i}e_{i+1}=he_ihy_{i} =  e_2\cdots e_{\lfloor\frac{n+1}{2}\rfloor} h e_2\cdots e_{\lfloor\frac{n}{2}\rfloor}h(xh)^2$, 
for $2\leqslant i \leqslant \lfloor\frac{n-1}{2}\rfloor$; 

\item $(\bar V'_0)$ $h^2=1$; 

\item $(\bar V'_1)$ $he_{\lfloor\frac{n+1}{2}\rfloor}=e_{\lfloor\frac{n+1}{2}\rfloor}h$, if $n$ is odd;  \\ 
$hx_i=hx^{n-i-1}hx_ix^{i-1}h$ and $hy_i=hx^{i-1}hy_ix^{n-i-1}h$, for $1\leqslant i\leqslant \lfloor\frac{n-1}{2}\rfloor$; 

\item $(\bar V'_2)$ $e_2\cdots e_{\lfloor\frac{n+1}{2}\rfloor} h e_2\cdots e_{\lfloor\frac{n}{2}\rfloor}(hx)^2 = x^{n-1}$. 

\end{description}

Thus, we have: 

\begin{theorem}\label{rankpresmdi} 
The monoid $\MDI_n$ is defined by the presentation $\langle \bar B' \mid \bar V'\rangle$ on $2+3\lfloor\frac{n-1}{2}\rfloor$ generators 
and $2n^2+\frac{7-(-1)^n}{4}n-2(-1)^n-1$ relations. 
\end{theorem} 

\section{Presentations for $\OPDI_n$}\label{presopdi}

As in both previous sections, we first determine a presentation for $\OPDI_n$ on an extended set of generators, namely, with $n+\lfloor\frac{n-1}{2}\rfloor+1$ generators, 
and then, through Tietze transformations, we deduce another presentation for $\OPDI_n$ on a minimum size set of generators, i.e. on $2+\lfloor\frac{n-1}{2}\rfloor$ generators. 

\smallskip 

Consider the alphabet $D=\{g,e_1,\ldots,e_n,x_1,\ldots,x_{\lfloor\frac{n-1}{2}\rfloor}\}$ and 
let $\psi:D^*\longrightarrow \OPDI_n$ be the homomorphism of monoids that extends the mapping 
$D\longrightarrow \OPDI_n$
defined by
$$
g\longmapsto g, \quad e_i\longmapsto e_i, \mbox{~for $1\leqslant i\leqslant n$}, 
\quad x_j\longmapsto x_j, \mbox{~for $1\leqslant j\leqslant \lfloor\frac{n-1}{2}\rfloor$}. 
$$
Let $Q$ be the set formed by the following 
$\frac{1}{2}(3n^2+(1-(-1)^n)n+3-\frac{1+(-1)^n}{2})$ monoid relations:  

\begin{description}

\item $(Q_1)$ $g^n=1$; 

\item $(Q_2)$ $e_{i}^{2}=e_{i}$,  for $1\leqslant i\leqslant n$; 

\item $(Q_3)$ $e_ie_j=e_je_i$, for $1\leqslant i<j\leqslant n$;

\item $(Q_4)$ $ge_1=e_ng$ and $ge_{i+1}=e_ig$, for $1\leqslant i\leqslant n-1$;

\item $(Q_5)$ $ge_1\cdots e_n=e_1\cdots e_n$;  

\item $(Q_6)$ $e_1x_i=x_ie_1=g^{n-2i}e_1\cdots e_{n-i}e_{n-i+2}\cdots e_n$,  for $1\leqslant i \leqslant \lfloor\frac{n-1}{2}\rfloor$;

\item $(Q_7)$ $x_{i}e_{j}=x_{i}$,  for $1\leqslant i \leqslant \lfloor\frac{n-1}{2}\rfloor$, $2\leqslant j\leqslant n$ and $j\ne n-i+1$;

\item $(Q_8)$ $e_{j}x_{i}=x_{i}$,  for $1\leqslant i \leqslant \lfloor\frac{n-1}{2}\rfloor$, $2\leqslant j\leqslant n$ and $j\ne i+1$; 

\item $(Q_9)$ $x_{i}e_{n-i+1}=e_{i+1}x_{i} = e_2\cdots e_n$, for $1\leqslant i \leqslant \lfloor\frac{n-1}{2}\rfloor$;

\item $(Q_{10})$ $(x_ig^i)^2=e_2\cdots e_ie_{i+2}\cdots e_n$, for $1\leqslant i \leqslant \lfloor\frac{n-1}{2}\rfloor$. 

\end{description}

Our aim is to show that the monoid $\OPDI_n$ is defined by the presentation $\langle D \mid Q\rangle$. 
As in Section \ref{presodi}, we will make use of results of \cite{Fernandes:2022sub}, this time in view to applying Proposition \ref{provingpresentation}. 

\smallskip 

We begin by noticing that it is a routine matter to check:  

\begin{lemma}\label{genrelopdi}
The set of generators $\{g,e_1,e_2,\ldots,e_n,x_1,\ldots,x_{\lfloor\frac{n-1}{2}\rfloor}\}$ 
of $\OPDI_n$ satisfies (via $\psi$) all relations from $Q$.
\end{lemma}

Observe that, as a consequence of the previous lemma, if $u,v\in D^*$ are such that the relation $u=v$ is a consequence of $Q$, 
then $u\psi=v\psi$. 

\smallskip 

Now, let us recall that the cyclic inverse monoid $\CI_n$ is generated by $\{g,e_1\}$ (see \cite{Fernandes:2022sub}) and so by $\{g,e_1,\ldots,e_n\}$. 

Let us consider the alphabet $D_0=\{g,e_1,\ldots,e_n\}$ and denote by $Q_0$ the subset of $Q$ consisting of relations $Q_1$ to $Q_5$.  
Then, we have: 

\begin{proposition}[{\cite[Theorem 2.6]{Fernandes:2022sub}}]\label{firstpresci} 
The monoid $\CI_n$ is defined by the presentation $\langle D_0 \mid Q_0\rangle$ on $n+1$ generators and $\frac{1}{2}(n^2+3n+4)$ relations. 
\end{proposition} 

To prove this result, the author used the property given by the following lemma, which we will also use here: 

\begin{lemma}[{\cite[Lemma 2.4]{Fernandes:2022sub}}]\label{pre1} 
Let $u\in D_0^*$. Then, there exist $0\leqslant m\leqslant n-1$, $1\leqslant i_1 < \cdots < i_k\leqslant n$ and $0\leqslant k\leqslant n$ such that 
the relation $u=g^m e_{i_1}\cdots e_{i_k}$ is a consequence of relations $Q_1$ to $Q_4$. 
\end{lemma}

Observe that, it is easy to show that,  for $1\leqslant i \leqslant \lfloor\frac{n-1}{2}\rfloor$, the relations 
\begin{equation}\label{eigm}
e_i g^m = \left\{\begin{array}{ll}
g^m e_{i+m} & \mbox{if $0\leqslant m\leqslant n-i$}\\
 g^m e_{i+m-n} & \mbox{if $n-i+1\leqslant m\leqslant n-1$}
\end{array}\right. 
\quad\text{and}\quad 
g^m e_i = \left\{\begin{array}{ll}
e_{i-m} g^m & \mbox{if $0\leqslant m\leqslant i-1$}\\
e_{n+i-m} g^m & \mbox{if $i\leqslant m\leqslant n-1$} 
\end{array}\right. 
\end{equation}
are consequences of $Q_4$. 

Combining (\ref{eigm}) with Lemma \ref{pre1}, we immediately obtain the \textit{symmetric} result of the latter one:

\begin{lemma}\label{pre2} 
Let $u\in D_0^*$. Then, there exist $0\leqslant m\leqslant n-1$, $1\leqslant i_1 < \cdots < i_k\leqslant n$ and $0\leqslant k\leqslant n$ such that 
the relation $u=e_{i_1}\cdots e_{i_k}g^m$ is a consequence of relations $Q_1$ to $Q_4$. 
\end{lemma} 

From now on, we denote the congruence $\rho_Q$ of $D^*$ again by $\approx$. 

\smallskip 
 
By Lemma \ref{r11r12}, we can conclude:

\begin{lemma}\label{xixj}
For all $1\leqslant i,j \leqslant \lfloor\frac{n-1}{2}\rfloor$, $x_ix_j\approx e_2\cdots e_n$.
\end{lemma}

Next, we prove a series of lemmas.  

\begin{lemma}\label{xiuxj} 
Let $u\in D_0^*$ and let $1\leqslant i,j \leqslant \lfloor\frac{n-1}{2}\rfloor$. 
Then, there exists $v\in D_0^*\cup x_iD_0^*\cup D_0^*x_j$ such that $x_iux_j\approx v$. 
\end{lemma} 
\begin{proof}
By Lemma \ref{pre1}, there exist $0\leqslant m\leqslant n-1$, $1\leqslant i_1 < \cdots < i_k\leqslant n$ and $0\leqslant k\leqslant n$ such that 
$u\approx g^m e_{i_1}\cdots e_{i_k}$. 

If $i_1=1$ (with $k>0$) then 
$$
x_iux_j 
\approx x_ig^m e_{i_2}\cdots e_{i_k}e_1x_j
\approx  x_ig^m e_{i_2}\cdots e_{i_k}g^{n-2j}e_1\cdots e_{n-j}e_{n-j+2}\cdots e_n \in x_iD_0^*, 
$$
by $Q_3$ and $Q_6$. 

If $j+1\in\{i_1,\ldots,i_k\}$ (with $k>0$) then, being $1\leqslant\ell\leqslant k$ such that $j+1=i_\ell$, we have 
$$
x_iux_j 
\approx x_ig^m e_{i_1}\cdots e_{i_\ell-1}e_{i_\ell+1}\cdots e_{i_k}e_{j+1}x_j 
\approx x_ig^m e_{i_1}\cdots e_{i_\ell-1}e_{i_\ell+1}\cdots e_{i_k} e_2\cdots e_n \in x_iD_0^*,  
$$
by $Q_3$ and $Q_9$. 

Now, suppose that either $k=0$ or $i_1>1$ and $j+1\not\in\{i_1,\ldots,i_k\}$ (with $k>0$).  
Then, by $Q_8$, $e_{i_1}\cdots e_{i_k}x_j\approx x_j$ and so $x_iux_j\approx x_ig^mx_j$. 

If $m=0$ then, by Lemma \ref{xixj}, $x_iux_j\approx x_ix_j\approx e_2\cdots e_n \in D_0^*$. So, suppose that $m>0$. 

If $m\neq i$ then $n-i+1\neq n-m+1\geqslant 2$ and so 
$$
x_iux_j\approx x_ie_{n-m+1}g^mx_j \approx x_ig^m e_1x_j \approx x_ig^mg^{n-2j}e_1\cdots e_{n-j}e_{n-j+2}\cdots e_n \in x_iD_0^*, 
$$
by $Q_7$, (\ref{eigm}) and $Q_6$. 

If $m\neq j$ then $j+1\neq m+1\geqslant 2$ and so 
$$
x_iux_j\approx x_ig^me_{m+1}x_j \approx x_ie_1 g^mx_j \approx g^{n-2i}e_1\cdots e_{n-i}e_{n-i+2}\cdots e_n g^mx_j \in D_0^*x_j, 
$$
by $Q_8$, (\ref{eigm}) and $Q_6$. 

Finally, if $m=i=j$ then 
$$
x_iux_j\approx x_ig^ix_i \approx x_ig^ix_i g^i g^{n-i} =(x_i g^i)^2 g^{n-i} \approx e_2\cdots e_ie_{i+2}\cdots e_n g^{n-i}\in D_0^*, 
$$
by $Q_1$ and $Q_{10}$, as required. 
\end{proof}

\begin{lemma}\label{ww} 
Let $w\in D^*$.  
Then, there exists $w'\in D_0^*\cup D_0^*x_1D_0^*\cup\cdots\cup D_0^*x_{\lfloor\frac{n-1}{2}\rfloor}D_0^*$ such that $w\approx w'$. 
\end{lemma} 
\begin{proof}
We proceed by induction on the number of occurrences of the letters $x_1, \ldots, x_{\lfloor\frac{n-1}{2}\rfloor}$ in a word $w \in D^*$. 

If $w \in D^*$ has no occurrences of the letters $x_1, \ldots, x_{\lfloor\frac{n-1}{2}\rfloor}$ then $w\in D_0^*$ and so there is nothing to prove. 

Hence, for $k\geqslant1$, suppose that the lemma is valid for all words in $D^*$ with $k-1$ occurrences of the letters $x_1, \ldots, x_{\lfloor\frac{n-1}{2}\rfloor}$. 

Let $w$ be a word of $D^*$ with $k$ occurrences of the letters $x_1, \ldots, x_{\lfloor\frac{n-1}{2}\rfloor}$. 
Then $w=u_0x_{i_1}u_1\cdots x_{i_{k-1}}u_{i_{k-1}}x_{i_k}u_k$, 
for some $u_0,u_1,\ldots,u_k \in D_0^*$ and $1\leqslant i_1,\ldots,i_k\leqslant \lfloor\frac{n-1}{2}\rfloor$. 
Hence, by induction hypothesis, there exists $u'\in D_0^*\cup D_0^*x_1D_0^*\cup\cdots\cup D_0^*x_{\lfloor\frac{n-1}{2}\rfloor}D_0^*$ 
such that $u_0x_{i_1}u_1\cdots x_{i_{k-1}}u_{i_{k-1}}\approx u'$ and so $w\approx u'x_{i_k}u_k$. 

If $u'\in D_0^*$ then the proof is finished. So, suppose that $u'=u'_0x_iu'_1$, for some $u'_0,u'_1\in D_0^*$ and some $1\leqslant i\leqslant \lfloor\frac{n-1}{2}\rfloor$. 
Thus, by Lemma \ref{xiuxj}, there exists $v\in D_0^*\cup x_iD_0^*\cup D_0^*x_{i_k}$ such that $x_iu'_1x_{i_k}\approx v$, whence 
$w\approx u'_0x_iu'_1x_{i_k}u_k \approx u'_0vu_k \in D_0^*\cup D_0^*x_iD_0^*\cup D_0^*x_{i_k}D_0^*$, as required. 
\end{proof}

\begin{lemma}\label{maindpdi} 
Let $w\in D^*$.  
Then, there exists $u\in D_0^*$ such that $w\approx u$ 
or there exist $1\leqslant i \leqslant \lfloor\frac{n-1}{2}\rfloor$ and $0\leqslant r,s\leqslant n-1$ such that $w\approx g^rx_ig^s$. 
\end{lemma} 
\begin{proof}
First, by Lemma \ref{ww}, take $w'\in D_0^*\cup D_0^*x_1D_0^*\cup\cdots\cup D_0^*x_{\lfloor\frac{n-1}{2}\rfloor}D_0^*$ such that $w\approx w'$. 
If $w'\in D_0^*$ then the proof is finished. So, suppose that $w'=ux_iv$, for some $u,v\in D_0^*$ and some $1\leqslant i\leqslant \lfloor\frac{n-1}{2}\rfloor$. 
Then, by Lemmas \ref{pre1} and \ref{pre2}, $u\approx g^r e_{i_1}\cdots e_{i_k}$ and $v\approx e_{j_1}\cdots e_{j_\ell}g^s$, 
for some $0\leqslant r,s\leqslant n-1$, $1\leqslant i_1 < \cdots < i_k\leqslant n$, $1\leqslant j_1 < \cdots < j_\ell\leqslant n$ and $0\leqslant k,\ell\leqslant n$.  
Hence $w\approx g^r e_{i_1}\cdots e_{i_k} x_i e_{j_1}\cdots e_{j_\ell}g^s$. 

If $i_1=1$ (with $k>0$) then, by $Q_3$ and $Q_6$,  
$$
w\approx g^r e_{i_2}\cdots e_{i_k} e_1x_i e_{j_1}\cdots e_{j_\ell}g^s 
\approx g^r e_{i_2}\cdots e_{i_k}   g^{n-2i}e_1\cdots e_{n-i}e_{n-i+2}\cdots e_n     e_{j_1}\cdots e_{j_\ell}g^s \in D_0^*. 
$$

If $i+1\in\{i_1,\ldots,i_k\}$ (with $k>0$) then, being $1\leqslant p\leqslant k$ such that $i+1=i_p$, we have 
$$
w\approx g^r e_{i_1}\cdots e_{i_p-1}e_{i_p+1}\cdots e_{i_k} e_{i+1} x_i  e_{j_1}\cdots e_{j_\ell}g^s 
\approx g^r e_{i_1}\cdots e_{i_p-1}e_{i_p+1}\cdots e_{i_k} e_2\cdots e_n e_{j_1}\cdots e_{j_\ell}g^s  \in D_0^*,  
$$
by $Q_3$ and $Q_9$. 

If $j_1=1$ (with $\ell>0$) then, by $Q_6$,  
$$
w\approx g^r e_{i_1}\cdots e_{i_k} x_i e_1e_{j_2}\cdots e_{j_\ell}g^s 
\approx g^r e_{i_1}\cdots e_{i_k}   g^{n-2i}e_1\cdots e_{n-i}e_{n-i+2}\cdots e_n   e_{j_2}\cdots e_{j_\ell}g^s \in D_0^*. 
$$

If $n-i+1\in\{j_1,\ldots,j_\ell\}$ (with $\ell>0$) then, being $1\leqslant q\leqslant \ell$ such that $n-i+1=j_q$, we have 
$$
w\approx g^r e_{i_1}\cdots  e_{i_k}  x_ie_{n-i+1}  e_{j_1}\cdots e_{j_q-1}e_{j_q+1}\cdots e_{j_\ell}g^s 
\approx g^r e_{i_1}\cdots  e_{i_k}  e_2\cdots e_n  e_{j_1}\cdots e_{j_q-1}e_{j_q+1}\cdots e_{j_\ell}g^s   \in D_0^*,  
$$
by $Q_3$ and $Q_9$. 

Finally, suppose that none of the four previous cases occurs. 
Then, by $Q_7$ and $Q_8$, $e_{i_1}\cdots e_{i_k}x_i e_{j_1}\cdots e_{j_\ell} \approx x_i$ and so $w\approx g^rx_ig^s$, as required.  
\end{proof}

For $1\leqslant i \leqslant \lfloor\frac{n-1}{2}\rfloor$ and $0\leqslant r,s\leqslant n-1$, let us consider the transformation $g^rx_ig^s$. 
It is a routine matter to check that 
$$
g^rx_ig^s=
\left\{\begin{array}{ll}
\begin{pmatrix}
              1 & 1+i \\
              1+s & a 
\end{pmatrix} & \text{if $r=0$} \\ 
\begin{pmatrix}
              n-r+1 & n-r+i+1 \\
              1+s & a 
\end{pmatrix} & \text{if $r>0$ and $1\leqslant i\leqslant r-1$} \\ 
\begin{pmatrix}
              i-r+1 & n-r+1 \\
              a & 1+s 
\end{pmatrix} & \text{if $r>0$ and $r\leqslant i\leqslant\lfloor\frac{n-1}{2}\rfloor$},    
\end{array}\right. 
$$
with 
$$
a = \left\{\begin{array}{ll}
            s-i+1 & \mbox{if $1\leqslant i\leqslant s$} \\
            n+s-i+1 & \mbox{if $s+1\leqslant i\leqslant\lfloor\frac{n-1}{2}\rfloor$}. 
           \end{array}\right.
$$
Hence, it is easy to show that 
\begin{equation}\label{rs0}
\mbox{$x_i=g^rx_ig^s$ if and only if $r=s=0$.} 
\end{equation}

Now, recall that $\{g,e_1\}$ generates $\CI_n$ and 
that $\{g,e_1,x_1,x_2,\ldots,x_{\lfloor\frac{n-1}{2}\rfloor}\}$ is a minimum size generating set of $\OPDI_n$. 
Therefore, noticing also that $g^n=1$, we have 
\begin{equation}\label{notinci}
\mbox{$g^rx_ig^s\not\in\CI_n$, for all $1\leqslant i \leqslant \lfloor\frac{n-1}{2}\rfloor$ and $r,s\geqslant 0$}, 
\end{equation}
and 
\begin{equation}\label{xjgrxigs}
\mbox{$x_j=g^rx_ig^s$, with $1\leqslant i,j \leqslant \lfloor\frac{n-1}{2}\rfloor$ and $r,s\geqslant 0$, implies $i=j$}.  
\end{equation}

We are now in a position to prove our first objective of this section.

\begin{theorem}\label{firstpresopdi} 
The monoid $\OPDI_n$ is defined by the presentation $\langle D\mid Q\rangle$ on $n+\lfloor\frac{n-1}{2}\rfloor+1$ generators 
and $\frac{1}{2}(3n^2+(1-(-1)^n)n+3-\frac{1+(-1)^n}{2})$ relations. 
\end{theorem} 
\begin{proof}
Given Lemma \ref{genrelopdi}, by Proposition \ref{provingpresentation}, 
it remains to prove that $w_1\approx w_2$ for all words $w_1,w_2\in D^*$ such that $w_1\psi=w_2\psi$. 
So, let $w_1,w_2\in D^*$ be such that $w_1\psi=w_2\psi$. 

By Lemma \ref{maindpdi}, for $k\in\{1,2\}$, there exists $u_k\in D_0^*$ such that $w_k\approx u_k$ (and then $w_k\psi= u_k\psi\in\CI_n$)
or there exist $1\leqslant i_k \leqslant \lfloor\frac{n-1}{2}\rfloor$ and $0\leqslant r_k,s_k\leqslant n-1$ such that $w_k\approx g^{r_k}x_{i_k}g^{s_k}$ 
(and then,  by (\ref{notinci}), $w_k\psi= (g^{r_k}x_{i_k}g^{s_k})\psi=g^{r_k}x_{i_k}g^{s_k}\not\in\CI_n$). 
Since $w_1\psi=w_2\psi$, we can only have:  ({\sc case 1}) $w_1\approx u_1$ and  $w_2\approx u_2$, for some $u_1,u_2\in D_0^*$; 
or ({\sc case 2}) $w_1\approx g^{r_1}x_{i_1}g^{s_1}$ and $w_2\approx g^{r_2}x_{i_2}g^{s_2}$, 
for some $1\leqslant i_1,i_2 \leqslant \lfloor\frac{n-1}{2}\rfloor$ and $0\leqslant r_1,s_1,r_2,s_2\leqslant n-1$. 

If {\sc case 1} occurs, then $u_1\psi=w_1\psi=w_2\psi=u_2\psi$ and so $u_1\rho_{Q_0}u_2$, 
since $\langle D_0 \mid Q_0\rangle$ is a presentation of $\CI_n$, by Proposition \ref{firstpresci},  
which implies that $u_1\approx u_2$ and thus $w_1\approx w_2$. 

Now, suppose we have {\sc case 2}. 
Then $g^{r_1}x_{i_1}g^{s_1}=w_1\psi=w_2\psi=g^{r_2}x_{i_2}g^{s_2}$ and so 
$x_{i_1}=g^{r_2-r_1+r_3}x_{i_2}g^{s_2-s_1+s_3}$, 
where 
$$
r_3=\left\{\begin{array}{ll} 
0 & \mbox{if $r_1\leqslant r_2$}\\
n & \mbox{if $r_1> r_2$}
\end{array}\right. 
\quad \text{and} \quad 
s_3=\left\{\begin{array}{ll} 
0 & \mbox{if $s_1\leqslant s_2$}\\
n & \mbox{if $s_1> s_2$}, 
\end{array}\right. 
$$
whence $i_1=i_2$, by (\ref{xjgrxigs}). 
Thus, it follows by (\ref{rs0}) that $r_2-r_1+r_3=0=s_2-s_1+s_3$. 
If $r_3=n$ then $n=r_1-r_2$, which is a contradiction, since $0\leqslant r_1,r_2\leqslant n-1$. 
Therefore $r_3=0$ and, analogously, $s_3=0$, whence $r_1=r_2$ and $s_1=s_2$. 
Thus $w_1\approx g^{r_1}x_{i_1}g^{s_1}=g^{r_2}x_{i_2}g^{s_2}\approx w_2$, as required. 
\end{proof}

\medskip 

Next, by using Tietze transformations and applying Proposition \ref{tietze}, we deduce from the previous presentation for $\OPDI_n$ 
a new one on the minimum size set of generators $\{g,e_1,x_1,x_2,\ldots,x_{\lfloor\frac{n-1}{2}\rfloor}\}$ of $\OPDI_n$.  

\smallskip 

By noticing that $e_i=g^{n-i+1}e_1g^{i-1}$ for $2\leqslant i\leqslant n$ (as transformations), 
we proceed as follows: first, by applying T1, we add the relations $e_i=g^{n-i+1}e_1g^{i-1}$, for $2\leqslant i\leqslant n$; 
secondly, we apply T4 with each of the relations $e_i=g^{n-i+1}e_1g^{i-1}$ with $2\leqslant i\leqslant n$; 
finally, by using the relation $Q_1$, we simplify the new relations obtained, eliminating the trivial ones or those that are deduced from others. 
By performing this procedure for each of the sets of relations $Q_1$ to $Q_{10}$, 
it is a routine matter to check that we can obtain the following set $Q'$ of 
$\frac{1}{2}(3n^2-(3+(-1)^n)n+5-\frac{1+(-1)^n}{2})$ monoid relations on the alphabet $D'=\{g,e_1,x_1,x_2,\ldots,x_{\lfloor\frac{n-1}{2}\rfloor}\}$: 
\begin{description}
\item $(Q'_1)$ $g^n=1$; 

\item $(Q'_2)$ $e_1^2=e_1$; 

\item $(Q'_3)$ $e_1g^{n-j+i}e_1g^{n-i+j} = g^{n-j+i}e_1g^{n-i+j}e_1$, for $1\leqslant i<j\leqslant n$; 

\item $(Q'_5)$ $g(e_1g^{n-1})^n=(e_1g^{n-1})^n$; 

\item $(Q'_6)$ $e_1x_i=x_ie_1=g^{n-2i} (e_1g^{n-1})^{n-i-1} e_1 (e_1g^{n-1})^{i-1}$,  for $1\leqslant i \leqslant \lfloor\frac{n-1}{2}\rfloor$;

\item $(Q'_7)$ $x_{i} g^{n-j+1}e_1g^{j-1} =x_{i}$,  for $1\leqslant i \leqslant \lfloor\frac{n-1}{2}\rfloor$, $2\leqslant j\leqslant n$ and $j\ne n-i+1$;

\item $(Q'_8)$ $g^{n-j+1}e_1g^{j-1}x_{i}=x_{i}$,  for $1\leqslant i \leqslant \lfloor\frac{n-1}{2}\rfloor$, $2\leqslant j\leqslant n$ and $j\ne i+1$; 

\item $(Q'_9)$ $x_i g^ie_1g^{n-i} = g^{n-i}e_1g^i x_i = g^{n-1} (e_1g^{n-1})^{n-1}$, for $1\leqslant i \leqslant \lfloor\frac{n-1}{2}\rfloor$; 

\item $(Q'_{10})$ $(x_ig^i)^2 = g^{n-1} (e_1g^{n-1})^{i-2} e_1g^{n-2} (e_1g^{n-1})^{n-i-1}$, for $1\leqslant i \leqslant \lfloor\frac{n-1}{2}\rfloor$. 

\end{description}

Thus, we have: 

\begin{theorem}\label{rankpresopdi} 
The monoid $\OPDI_n$ is defined by the presentation $\langle D' \mid Q'\rangle$ on $2+\lfloor\frac{n-1}{2}\rfloor$ generators 
and $\frac{1}{2}(3n^2-(3+(-1)^n)n+5-\frac{1+(-1)^n}{2})$ relations. 
\end{theorem} 

\medskip 

\paragraph{Acknowledgments.} 
We would like to thank the anonymous referee for the valuable comments and suggestions that allowed us to improve this paper.

\bigskip

\lastpage


\begin{thebibliography}{00}

%
%
%

\bibitem{Araujo&al:2011}
J. Ara\'ujo, V.H. Fernandes, M.M. Jesus, V. Maltcev and J.D. Mitchell, 
Automorphisms of partial endomorphism semigroups, Publicationes Mathematicae Debrecen 79 (2011), 23--39. 

%

\bibitem{Cowan&Reilly:1995}
D.F. Cowan and N.R. Reilly, 
Partial cross-sections of symmetric inverse semigroups, 
Int. J. Algebra Computation 5 (1995), 259--287.

\bibitem{Delgado&Fernandes:2000}
M. Delgado and V.H. Fernandes,  
Abelian kernels of some monoids of injective partial transformations and an application,  
Semigroup Forum 61 (2000), 435--452.


\bibitem{Dimitrova&al:2022}
I. Dimitrova, V.H. Fernandes, J. Koppitz and T.M. Quinteiro, 
On three remarkable submonoids of the dihedral inverse monoid on a finite set, 
arXiv:2301.01519 (2023),  
https://doi.org/10.48550/arXiv.2301.01519. 

\bibitem{East:2006}
J. East, 
A presentation of the singular part of the symmetric inverse monoid, 
Commun. Algebra 34 (2006), 1671--1689.



\bibitem{Fernandes:1997}
V.H. Fernandes,
Semigroups of order preserving mappings on a finite chain: a new class of divisors, 
Semigroup Forum 54 (1997), 230--236. 

\bibitem{Fernandes:1998}
V.H. Fernandes,
Normally ordered inverse semigroups, 
Semigroup Forum 56 (1998), 418--433. 

\bibitem{Fernandes:2000}
V.H. Fernandes,
The monoid of all injective orientation preserving partial transformations on a finite chain,
Commun. Algebra 28 (2000), 3401--3426.

\bibitem{Fernandes:2001}
V.H. Fernandes,
The monoid of all injective order preserving partial transformations on a finite chain,
Semigroup Forum 62 (2001), 178-204.

\bibitem{Fernandes:2001b}
V.H. Fernandes,
A division theorem for the pseudovariety generated by semigroups of orientation preserving transformations on a finite chain, 
Commun. Algebra 29 (2001), 451--456


\bibitem{Fernandes:2008}
V.H. Fernandes, 
On divisors of pseudovarieties generated by some classes of full transformation semigroups, 
Algebra Colloq. 15 (2008), 581--588.

\bibitem{Fernandes:2022sub}
V.H. Fernandes,
On the cyclic inverse monoid on a finite set,
arXiv:2211.02155 (2022),  \\ 
https://doi.org/10.48550/arXiv.2211.02155. 

\bibitem{Fernandes&Gomes&Jesus:2004}
V.H. Fernandes, G.M.S. Gomes and M.M. Jesus,
Presentations for some monoids of injective partial transformations on a finite chain,
Southeast Asian Bull. Math. 28 (2004), 903--918.

\bibitem{Fernandes&Gomes&Jesus:2005}
V.H. Fernandes, G.M.S. Gomes and M.M. Jesus,
Congruences on monoids of order-preserving or order-reversing transformations on a finite chain, 
Glasg. Math. J. 47 (2005), 413--424.


\bibitem{Fernandes&Paulista:2022sub}
V.H. Fernandes and T. Paulista,
On the monoid of partial isometries of a cycle graph,
Turk. J. Math., 47 (2023), 1746--1760.


\bibitem{Fernandes&Santos:2019}
V.H. Fernandes and P.G. Santos, 
Endomorphisms of semigroups of order-preserving partial transformations, 
Semigroup Forum 99 (2019), 333--344. 

\bibitem{Ganyushkin&Mazorchuk:2003}
O. Ganyushkin and V. Mazorchuk, 
On the Structure of $\mathcal{IO}n$, 
Semigroup Forum 66 (2003), 455--483. 
 
\bibitem{Howie:1995}
J.M. Howie,
Fundamentals of Semigroup Theory,
Oxford, Oxford University Press, 1995.

%

\bibitem{Lallement:1979}
G. Lallement,
Semigroups and Combinatorial Applications,
John Wiley \& Sons, New York, 1979.

\bibitem{Li&Fernandes:2023}
D.B. Li and V.H. Fernandes, 
Endomorphisms of semigroups of oriented transformations, 
Semigroup Forum, 106 (2023), 184--210. 

%
%
%

\bibitem{Ruskuc:1995}
N. Ru\v{s}kuc,
Semigroup Presentations,
Ph.D. Thesis, University of St-Andrews, 1995.

\bibitem{GAP4}
The GAP~Group, \emph{GAP -- Groups, Algorithms, and Programming,
 Version 4.11.1}; 2021. \newline (https://www.gap-system.org)

\end{thebibliography}
\end{document}